\documentclass[letterpaper,11pt,reqno]{amsart} 

\RequirePackage[utf8]{inputenc}
\usepackage[portrait,margin=3cm]{geometry}
\usepackage{color}              
\usepackage[usenames,dvipsnames]{xcolor}
\usepackage{mathrsfs,soul}
\usepackage{hyperref}
\usepackage{amssymb,amsthm,amsmath,amsfonts,amsbsy,latexsym,dsfont}
\usepackage[textsize=small]{todonotes}       
\usepackage{graphicx}
\usepackage[numeric,initials,nobysame]{amsrefs}
\usepackage{upref,setspace}
\usepackage{enumerate}

\newtheorem{Lemma}{Lemma}[section]

\newtheorem{Theorem}{Theorem}[section]
\newtheorem{Assumption}{Assumption}[section]

\newtheorem{Remark}{Remark}[section]

\newtheorem{Corollary}{Corollary}[section]
\newtheorem{Definition}{Definition}[section]

\numberwithin{equation}{section}
\numberwithin{figure}{section}
\numberwithin{table}{section}

\newcommand{\PRMspace}{H}
\newcommand{\pathspace}{\mathbb{D}([0,T]:\bell_1^\downarrow \times \bell_1)}
\newcommand{\pathspacer}{\mathbb{D}([0,T]:\Rmb^{\infty} \times \Rmb^{\infty})}
\newcommand{\pathspacef}{\mathbb{D}([0,T]:\bell_1^\downarrow)}
\newcommand{\leb}{\mbox{Leb}}
\newcommand{\controlset}{\mathcal{A}}
\newcommand{\Xt}{{\mathbb{X}_t}}
\newcommand{\XT}{{\mathbb{X}_T}}
\newcommand{\rate}{{\mathcal{I}}}

\newcommand{\noi}{\noindent}

\newcommand{\ti}{\tilde}

\newcommand{\Cmb}{{\mathbb{C}}}
\newcommand{\Dmb}{{\mathbb{D}}}
\newcommand{\Emb}{{\mathbb{E}}}

\newcommand{\Mmb}{{\mathbb{M}}}
\newcommand{\Nmb}{{\mathbb{N}}}

\newcommand{\Pmb}{{\mathbb{P}}}

\newcommand{\Rmb}{{\mathbb{R}}}
\newcommand{\Smb}{{\mathbb{S}}}

\newcommand{\Vmb}{{\mathbb{V}}}

\newcommand{\Amc}{{\mathcal{A}}}
\newcommand{\Bmc}{{\mathcal{B}}}
\newcommand{\Cmc}{{\mathcal{C}}}

\newcommand{\Fmc}{{\mathcal{F}}}

\newcommand{\Imc}{{\mathcal{I}}}

\newcommand{\Mmc}{{\mathcal{M}}}

\newcommand{\Smc}{{\mathcal{S}}}
\newcommand{\Tmc}{{\mathcal{T}}}

\newcommand{\zero}{{\boldsymbol{0}}}
\newcommand{\one}{{\boldsymbol{1}}}

\newcommand{\Bbd}{{\boldsymbol{B}}}

\newcommand{\ebd}{{\boldsymbol{e}}}
\newcommand{\etabd}{{\boldsymbol{\eta}}}

\newcommand{\hbd}{{\boldsymbol{h}}}

\newcommand{\phibd}{{\boldsymbol{\phi}}}
\newcommand{\psibd}{{\boldsymbol{\psi}}}

\newcommand{\varphibd}{{\boldsymbol{\varphi}}}

\newcommand{\xbd}{{\boldsymbol{x}}}\newcommand{\Xbd}{{\boldsymbol{X}}}
\newcommand{\ybd}{{\boldsymbol{y}}}\newcommand{\Ybd}{{\boldsymbol{Y}}}

\newcommand{\zetabd}{{\boldsymbol{\zeta}}}
\newcommand{\bell}{{\boldsymbol{\ell}}}

\newcommand{\Amcbar}{{\bar{\Amc}}}

\newcommand{\etabar}{{\bar{\eta}}}

\newcommand{\Nbar}{{\bar{N}}}

\newcommand{\psibar}{{\bar{\psi}}}

\newcommand{\thetabar}{{\bar{\theta}}}

\newcommand{\varphibar}{{\bar{\varphi}}}

\newcommand{\Xbar}{{\bar{X}}}

\newcommand{\Ybar}{{\bar{Y}}}

\newcommand{\zetabar}{{\bar{\zeta}}}

\newcommand{\ahat}{{\hat{a}}}
\newcommand{\bhat}{{\hat{b}}}
\newcommand{\chat}{{\hat{c}}}

\newcommand{\Btil}{{\tilde{B}}}

\newcommand{\Dtil}{{\tilde{D}}}

\newcommand{\etatil}{{\tilde{\eta}}}

\newcommand{\varphitil}{{\tilde{\varphi}}}

\begin{document}

\title[Large deviation asymptotics for JSQ$(d_n)$]{Large Deviation Asymptotics for the Supermarket Model with Growing Choices}
\author{Amarjit Budhiraja and Ruoyu Wu}
\date{\today}
\begin{abstract}
We consider the Markovian supermarket model with growing choices, where jobs arrive at rate $n\lambda_n$ and each of $n$ parallel servers processes jobs in its queue at rate $1$.
Each incoming job joins the shortest among $d_n \in \{1,\dotsc,n\}$ randomly selected queues.
Under the assumption $d_n \to \infty$ and $\lambda_n \to \lambda \in (0,\infty)$ as $n\to \infty$, a large deviation principle (LDP) for the occupancy process is established in a suitable infinite-dimensional path space, and it is shown that the rate function is invariant with respect to the manner in which $d_n \to \infty$.
The LDP gives information on the rate of decay of probabilities of various types of rare events associated with the system.
We illustrate this by establishing explicit exponential decay rates for probabilities of large total number of jobs in the system. 
As a corollary, we also show that probabilities of certain rare events can indeed depend on the rate of $d_n \to \infty$.

\smallskip

\noi {\bf AMS 2020 subject classifications:}
60F10, 
60J74, 
34H05, 
90B15 

\smallskip

\noi{\bf Keywords:} calculus of variations, diminishing rates, discontinuous statistics,  infinite-dimensional Skorokhod problem, join-the-shortest-queue, jump-Markov processes in infinite dimensions, large deviations, load balancing, power of choice

\end{abstract}
\maketitle

\section{Introduction}
\label{sec:intro}
This work investigates the asymptotic behavior, under a large deviation scaling, of a class of randomized load balancing schemes in large-scale multi-server systems. We consider a system with $n$ parallel queues,  and jobs arriving according to a Poisson process with rate $n\lambda_n$, where $\lambda_n \to \lambda \in (0,\infty)$ as $n\to \infty$.  Each server processes jobs in its queue using the FIFO protocol and the service times are exponential with mean $1$. We assume that the inter-arrival times and service times are mutually independent. Each incoming job joins the shortest among $d_n$ queues, where $d_n \in \{1, \dotsc, n\}$, sampled uniformly at random without replacement.
This policy is commonly referred to as JSQ($d_n$), namely Join-the-Shortest-Queue-$d_n$ or the ``supermarket model''.

Two important special cases are: JSQ(1), where each job selects a queue uniformly at random, leading to $n$ independent M/M/1 queues; and JSQ($n$), where the job joins the shortest of all $n$ queues, referred to simply as JSQ. The latter scheme is known to achieve the optimal load balancing while the former is very easy to implement without need of any state information. The case of a fixed $d > 1$ is known as the ``power-of-$d$'' scheme.

It is well known from the works of  Mitzenmacher\cite{mitzenmacher2001power} and Vvedenskaya et al.\cite{vvedenskaya1996queueing}  that increasing $d$ from $1$ to $2$ greatly improves performance in terms of queue length distribution: the tail decays exponentially when $d = 1$ and superexponentially when $d = 2$. Subsequent studies have established diffusion limits for fixed $d$ (see, e.g., \cite{banerjee2018join, bramson2012asymptotic}). Several works have  explored the trade-offs between complexity and performance under different choices of $d$. See the comprehensive survey \cite{van2018scalable} for an overview of this general area.

Although fixed $d \geq 2$ schemes offer substantial gains over random assignment (i.e. $d=1$), they still fall short of the performance achieved by JSQ($n$). This motivates analyzing the regime where $d_n$ increases with $n$, which is the focus of this work. 

We are specifically interested in the large deviation behavior of the system as $n\to \infty$. 
A natural state descriptor for a JSQ($d_n$) system, at time instant $t \in [0,T]$,  is the infinite-dimensional state occupancy vector $\Xbd^n(t) = (X^n_0(t), X^n_1(t),\dotsc)$ where $X^n_i(t)$ corresponds to the proportion of queues which are of length $i$ or longer at time $t$. 
It was shown in \cite{BhamidiBudhirajaDewaskar2022near} that $\Xbd^n$ converges in $\pathspacef$, in probability, as $n \to \infty$, where $\pathspacef$ is the space of c\'{a}dl\'{a}g functions from $[0,T]$ to $\bell_1^{\downarrow}$ (here  $\bell_1^{\downarrow}$ is a closed subset of the Banach space $\bell_1$ -- the space of real absolutely summable sequences equipped with the usual norm -- cf. Section \ref{sec:notat}), to a deterministic limit, whenever $d_n \to \infty$. Furthermore this {\em law of large numbers} (LLN) limit does not depend on the manner in which $d_n \to \infty$. Previously \cite{mukboretal} had shown that $\Xbd^n$ is tight in the above path space and any limit point satisfies the same system of {\em fluid limit equations} irrespective of how $d_n$ approached $\infty$.

The latter paper also showed that when $\frac{d_n}{\sqrt{n}\log n} \to \infty$, and $\lambda_n \to 1$, then with a suitable centering and normalization the state occupancy process is asymptotically described by a two-dimensional Gaussian process which  previously had been shown to be the limit of these fluctuations in the case $d_n=n$ in \cite{eschenfeldt2018join}. 
In order to differentiate the asymptotic behavior of JSQ($d_n$) for $d_n <n$ from that of JSQ($n$), the paper \cite{BhamidiBudhirajaDewaskar2022near} investigated diffusion approximations for the suitably centered and normalized state occupancy process  in the critical
regime (i.e., when $\lambda_n \to 1$ in a suitable manner) that allow for possibly a slower growth of $d_n$
than that permitted by the results in \cite{mukboretal}. This paper showed that, in contrast to the LLN behavior which is insensitive to the manner in which $d_n \to \infty$, the   diffusion
limit depends crucially on the rate of growth of $d_n$ and provided distinct explicit characterizations for the limiting fluctuations in the three regimes: $d_n/\sqrt{n} \to 0$, $d_n/\sqrt{n} \to c \in (0,\infty)$, and $d_n/\sqrt{n} \to \infty$.

In this work we are interested in the large deviation behavior of the state occupancy process $\Xbd^n$ in the JSQ($d_n$) system. Throughout we assume that $\Xbd^n(0) = \xbd^n$, where $\xbd^n \in \bell_1^{\downarrow}$ and, for some $\xbd \in \bell_1^{\downarrow}$, $\xbd^n \to \xbd$ in $\bell_1$.  For the case $d_n=n$, a large deviation principle (LDP) for $\Xbd^n$ in 
$\mathbb{D}([0,T]:\Rmb^{\infty})$ was established in \cite{BudhirajaFriedlanderWu2021many}. 
As noted there, the key model features that present technical challenges in the analysis of this large deviation problem are Markovian dynamics with discontinuous statistics, a diminishing rate property of the jump rates, and the infinite dimensionality of the state space; see \cite[Section 1]{BudhirajaFriedlanderWu2021many} for a detailed discussion of these points.
The goal of the current work is two-fold: first to allow for general sequences $d_n \to \infty$, and second to strengthen the topology for the LDP from $\mathbb{D}([0,T]:\Rmb^{\infty})$ to $\mathbb{D}([0,T]:\bell_1^{\downarrow})$.

One of the key observations in the analysis of \cite{BudhirajaFriedlanderWu2021many} was that when $d_n=n$, one can introduce an infinite-dimensional Skorokhod map $\Gamma_{\infty}: \mathbb{D}([0,T]:\Rmb^{\infty}) \to \mathbb{D}([0,T]:(-\infty, 1]^{\infty})$ (see Definition \ref{def:Smap}) and a {\em free process} $\Ybd^n$, associated with the occupancy process $\Xbd^n$, with sample paths in $\mathbb{D}([0,T]:\Rmb^{\infty})$ such that $\Xbd^n=\Gamma_{\infty}(\Ybd^n)$. 
Using this relation, \cite{BudhirajaFriedlanderWu2021many} in fact established a LDP for the pair $(\Xbd^n, \Ybd^n)$ in $\mathbb{D}([0,T]:\Rmb^{\infty} \times \Rmb^{\infty})$.  
In the general setting of $d_n <n$, one can once more associate a similar free process $\Ybd^n$ with $\Xbd^n$ (see equations \eqref{eq:Y1}--\eqref{eq:Yi}), 
however in this case one does not in general have the property $\Xbd^n = \Gamma_{\infty}(\Ybd^n)$ (see Remark \ref{rem:2.2}). 
Roughly speaking, this difficulty arises from the feature that, when $d_n<n$, one may have arrivals to queues of length $j$ or higher at instants $t$ even if $X^n_{j-1}(t) <1$. 
This behavior is impossible in the JSQ($n$) system. This difficulty requires us to develop a different analysis, particularly for the proof of the large  deviation upper bound. Our main result (Theorem \ref{thm:mainResult}) shows that the pair process $(\Xbd^n, \Ybd^n)$ satisfies a LDP in $\pathspace$ with the same rate function as in \cite{BudhirajaFriedlanderWu2021many} (with definition restricted to this smaller space). 
The main observation here is that although the controlled occupancy processes and the associated free processes that arise in the large deviation analysis cannot be related through the Skorokhod map $\Gamma_{\infty}$, their weak limits $(\Xbd^*, \Ybd^*)$ are indeed related in this manner (namely, $\Xbd^* = \Gamma_{\infty}(\Ybd^*)$; see Lemma \ref{lem:convergence}). 
This invariance result at the large deviation scaling is in sharp contrast to the behavior under the diffusive scaling studied in \cite{BhamidiBudhirajaDewaskar2022near} (for the critical regime, $\lambda_n \to 1$), which was discussed in the previous paragraph. The reason for this can be seen from the key term $\beta_n$ that appears in the evolution equation for the 
occupancy process (see \eqref{eq:beta} and \eqref{eq:state-process}). For the LLN and LDP analysis one only needs to understand the behavior of $\beta_n(x)$ for a fixed $x <1$ and as long as $d_n \to \infty$, for such $x$, $\beta_n(x) \to 0$. In contrast, for the study of the system under diffusion scaling one needs to analyze the properties of $\beta_n$ in $O(n^{-1/2})$ neighborhoods of $1$ which can lead to complex asymptotic limiting behavior that depends intricately on the rates at which $d_n \to \infty$ and $\lambda_n \to 1$.

The strengthening of the LDP from the space $\mathbb{D}([0,T]:\Rmb^{\infty} \times \Rmb^{\infty})$ to $\pathspace$ also requires additional work, specifically in the tightness proofs that are needed both for the upper and lower bounds. 
One basic obstacle is that the infinite-dimensional Skorokhod map, which is a Lipschitz function from $\mathbb{D}([0,T]:\Rmb^{\infty})$ to itself (see \cite[Lemma 2.2]{BudhirajaFriedlanderWu2021many}) is not Lipschitz as a map from $\mathbb{D}([0,T]:\bell_1)$ to itself. 
We overcome this difficulty by reducing analyses to that of finite-dimensional Skorokhod maps (see e.g.\ the proofs of Lemma 
\ref{lem:tightCompact} and Lemma 
\ref{lem:special-well-posedness}) for which the Lipschitz property is available (see Remarks \ref{rmk:Skorokhod-map} and \ref{rem:2.3}).

The lower bound analysis requires additional care.  
Indeed the most technically demanding part of the proof of the LDP in \cite{BudhirajaFriedlanderWu2021many} was a certain uniqueness result for a system of equations for continuous $\Rmb^{\infty} \times \Rmb^{\infty}$-valued trajectories $(\zetabd, \psibd)$ involving  certain control sequences $\varphibd$ (see Lemma 5.1 in \cite{BudhirajaFriedlanderWu2021many} and also Lemma \ref{lem:uniqueness} of the current work).  
This required a series of delicate approximations to a given pair of trajectories $(\zetabd, \psibd)$ that were suitably close with respect to the metric on 
$\mathbb{D}([0,T]:\Rmb^{\infty} \times \Rmb^{\infty})$. 
Although the same approximation scheme works in the current setting, one needs to ensure that the errors in the approximations are controlled with respect to the more demanding metric on $\pathspace$.

One of the advantages of establishing a LDP for $\Xbd^n$ in
$\mathbb{D}([0,T]:\bell_1^{\downarrow})$ is that it immediately yields a LDP for the process $Z^n$ of  total number of jobs in the system in $\mathbb{D}([0,T]:\Rmb)$. This follows on noting that
$Z^n(t) = \sum_{k=1}^{\infty} X^n_k(t)$ and applying the contraction principle.  One can similarly establish a LDP of related quantities, such as the process of total number of jobs in queues of lengths $k$ or higher.  Although in general the variational problems governing these LDP results are not tractable for explicit calculations, in some cases, by exploiting special features of the associated calculus of variations problems, one can obtain more information.  We illustrate this in Theorem \ref{thm:large-customers} by considering the setting where $\lambda=1$ and $x_1=1$ (i.e. asymptotically all servers are busy).
Roughly speaking this result shows that, for a given $\varepsilon>0$, denoting by $G^n_\varepsilon$ the event that  the number of  jobs  in the system 
at some instant $t \in [0,T]$ is at least $n\varepsilon$ more than the initial number of jobs, namely,
$$G^n_\varepsilon  := \{\|\Xbd^n(t)\|_1 > \|\xbd^n\|_1 + \varepsilon \mbox{ for some } t \in [0,T]\},
$$
we have, for large $n$ 
$$
\Pmb(G^n_\varepsilon) \approx
		\exp\left\{ - nT\ell\left( \frac{\frac{\varepsilon}{T}+\sqrt{4+(\frac{\varepsilon}{T})^2}}{2} \right) - nT\ell\left( \frac{-\frac{\varepsilon}{T}+\sqrt{4+(\frac{\varepsilon}{T})^2}}{2} \right)\right\},
$$
where $\ell$ is as defined in Section \ref{sec:notat}.
In particular this says that for large $n$ and $T$
$$
\Pmb(G^n_\varepsilon) \approx e^{-n\varepsilon^2/4T}.$$
For precise statement see Theorem \ref{thm:large-customers}.  As an immediate corollary of this result we obtain the asymptotic formula established in \cite[Theorem 2.5]{BudhirajaFriedlanderWu2021many},  for {\em buffer overflow} events $U^n_j$ and $V^n_j$, in the case $d_n=n$ and $\xbd^n = \xbd = (1, 0, \ldots)$ (see Corollary \ref{cor:long-queue}).  This result also illustrates the important point that although the LDP  is invariant  under the choice of the sequence $d_n \to \infty$, the asymptotic decay rate for specific events can indeed depend in an intricate manner on the rate at which $d_n \to \infty$. Specifically, in 
Remark \ref{rem:diffdec} we show that when $j=3$, the asymptotic exponential decay rate of $\Pmb(U^n_j)$ is strictly positive when $d_n = n$ and equals $0$ when $d_n = o(n)$, capturing the performance improvement in the former case in comparison with the latter case.

One key ingredient in the proof of Theorem \ref{thm:large-customers} is Lemma \ref{lem:special-well-posedness} which gives well-posedness of certain infinite system of equations with Skorokhod reflections and state feedback controls. 
Using this result we construct the {\em most likely state trajectory} associated with the event $G^n_{\varepsilon}$ given in terms of a suitably chosen feedback control (see Section \ref{sec:examples}, below the proof of Lemma \ref{lem:special-well-posedness}). 
Verifying that this is indeed the optimal trajectory is the most demanding part of this section and uses ideas from calculus of variations and exploits the convexity properties of the cost function $\ell$.

Finally, we remark that requiring $d_n \to \infty$ allows for some simplifications in the large deviation analysis.  
The large deviation behavior for the JSQ($d$) model, namely when $d_n=d \ge 2$ for all $n \in \Nmb$ is currently an open problem. 
One of the key challenges arises from the asymptotic behavior of $\beta_n$. 
When $d_n \to \infty$, $\beta_n(x) \to 0$ for all $x \in (0,1)$ whereas when $d_n \equiv d$, $\beta_n(x) \to x^d$, as $n \to \infty$. 
This introduces a non-trivial, nonlinear (since $d\ge 2$) state dependence and the current analysis that relies on properties of an infinite-dimensional Skorokhod map is not applicable. 
The main challenge is once more in the proof of the large deviation lower bound and in establishing a uniqueness result analogous to Lemma \ref{lem:uniqueness}. 
We leave this study for future work.

\subsection{Notation}
\label{sec:notat}

The following notation will be used. 
Fix $T \in (0,\infty)$. All stochastic processes will be considered over the time horizon $[0,T]$.  
Let $\Nmb_0 := \Nmb\cup\{0\}$, where $\Nmb$ is the set of all natural numbers.
Let $S$ be a Polish space.
The Borel $\sigma$-field on $S$ will be denoted as $\Bmc(S)$.
Denote by $\Dmb([0,T]:S)$ the collection of all maps from $[0,T]$ to $S$ that are right continuous and have left limits. 
This space is equipped with the usual 
Skorokhod topology. 
Similarly $\Cmb([0,T]:S)$ is the space of all continuous maps from $[0,T]$ to $S$ equipped with the uniform topology. 
A sequence of $\Dmb([0,T]:S)$-valued random variables is said to be $\Cmb$-tight if it is tight in $\Dmb([0,T]:S)$ and any weak limit point takes values in $\Cmb([0,T]:S)$ a.s.
The space of all continuous and bounded real-valued functions on $S$ will be denoted as $\Cmb_b(S)$.
For a bounded map $f \colon S \to \Rmb$, let $\|f\|_{\infty} := \sup_{s \in S} |f(s)|$.
Denote by $\bell_1$ the space of real sequences $\xbd := (x_1,x_2,\dotsc)$ such that $\|x\|_1 := \sum_{i=1}^\infty |x_i| < \infty$.
Let 
\begin{equation}
    \label{eq:ell-1-down}
    \bell_1^\downarrow := \{\xbd \in \bell_1 : x_i \ge x_{i+1} \mbox{ and } x_i \in [0,1] \mbox{ for all } i \in \Nmb \}
\end{equation}
be the space of non-increasing sequences in $\bell_1$ with values in $[0,1]$, equipped with the $\|\cdot\|_1$ norm.
Note that $\bell_1^\downarrow$ is a closed subset of $\bell_1$ and hence is a Polish space.
The $L^1$ norm on $\Rmb^m$,  $m \in \Nmb$, will also be denoted as $\|\cdot\|_1$.
Denote by $\Rmb^\infty$ the set of all real sequence $\xbd = (x_1,x_2,\dotsc)$ equipped with the product topology, which is metrized with
\begin{equation}
	\label{eq:d-infty}
	d_\infty(\xbd,\ybd) := \sum_{i=1}^\infty\frac{|x_i-y_i|\wedge 1}{2^i}, \quad \xbd,\ybd \in \Rmb^\infty.
\end{equation}
Let $\ell(z) := z\log(z)-z+1$ for $z \ge 0$. 
For $t \in [0,T]$, write $\Xt:=[0,t]\times[0,1]$.
\subsection{Organization}
The rest of this paper is organized as follows.
Section \ref{sec:model} introduces the state dynamics in terms of an infinite collection of Poisson random measures. 
It also gives an equivalent representation using certain free processes and regulator processes, which together asymptotically solve an infinite-dimensional Skorokhod problem. 
In Section \ref{sec:rateFunction}, properties of the solution map of this Skorokhod problem are summarized and the rate function that governs the LDP is introduced.
The main result, Theorem \ref{thm:mainResult}, is then given in Section \ref{sec:mainResult}. 
This section also presents Theorem \ref{thm:large-customers} which gives our main result on exponential decay rates for probabilities of large total number of customers waiting in the system, as an illustration of applications of Theorem \ref{thm:mainResult}.
Dependence of the decay rate for certain events on the rate at which $d_n \to \infty$ is shown in Corollary \ref{cor:long-queue} and Remark \ref{rem:diffdec}. 
Section \ref{sec:varWeakCon} introduces the main variational representation that is the starting point of our analysis and establishes preliminary tightness and limit characterization results that are used in both the Laplace upper bound and lower bound proofs. 
Proof of the Laplace upper bound (i.e.\ \eqref{eqn:LaplaceUpperBound}) is completed in Section \ref{sec:UpperBound} while the lower bound (i.e.\ \eqref{eqn:LaplaceLowerBound}) is taken up in Section \ref{sec:LowerBound} with some auxiliary arguments given in Appendix \ref{sec:appendix}.
Section \ref{sec:compactSets} shows that the function $\Imc$ introduced in Section \ref{sec:rateFunction} is indeed a rate function. 
The results of Sections \ref{sec:UpperBound}, \ref{sec:LowerBound}, and \ref{sec:compactSets} together complete the proof of Theorem
\ref{thm:mainResult}. 
Finally Section \ref{sec:examples} gives the proof of Theorem \ref{thm:large-customers}.
\section{Model}

\subsection{Model Description}
\label{sec:model}
We recall the setting from Section \ref{sec:intro}. For $n \in \Nmb$, fix $d_n\in \{1, \dotsc, n\}$.
Consider a system of $n$ parallel servers each maintaining its own queue.
Jobs arrive to a central dispatcher according to a Poisson process with rate $n\lambda_n$ where $\lambda_n \to\lambda$ for some $\lambda\in(0,\infty)$.
When a job enters the system, it joins the shortest queue among $d_n$ randomly selected queues (without replacement).
If there are multiple shortest queues, then the tie is broken uniformly at random. This is commonly referred to as the JSQ$(d_n)$ routing policy.
We assume throughout that $d_n \to \infty$ as $n\to \infty$.
Each server processes jobs in its queue using the FIFO protocol and the service times are exponential with mean $1$. 
We assume that the inter-arrival times and service times are mutually independent.
The state of the system at time $t$ can be represented as $\Xbd^n(t) = (X^n_0(t), X^n_1(t),\dotsc)$ where $X^n_i(t)$ corresponds to the proportion of queues which are of length $i$ or longer at time $t$.
Note that $X^n_i(t)\in[0,1]$ and $1=X^n_0(t)\geq X^n_1(t)\geq X^n_2(t)\geq \dotsc$ for all $t\in[0,T]$.

We will now give an evolution equation for the state process, which will be convenient for the large deviation analysis, in terms of a collection of Poisson random measures. 
For a locally compact metric space $\Smb$, let $\Mmc_{FC}(\Smb)$ represent the space of measures $\nu$ on $(\Smb,\Bmc(\Smb))$ such that $\nu(K)<\infty$ for every compact $K\in\Bmc(\Smb)$, equipped with the usual vague topology.
This topology can be metrized such that $\Mmc_{FC}(\Smb)$ is a Polish space (see \cite{budhiraja2016moderate, BudhirajaDupuis2019analysis} for one convenient metric).
A PRM $D$ on $\Smb$ with mean measure (or intensity measure) $\nu\in\Mmc_{FC}(\Smb)$ is an $\Mmc_{FC}$-valued random variable such that for each $\PRMspace\in\Bmc(\Smb)$ with $\nu(\PRMspace)<\infty$, $D(\PRMspace)$ is a Poisson random variable with mean $\nu(\PRMspace)$ and for disjoint $\PRMspace_1,\dotsc,\PRMspace_i\in\Bmc(\Smb)$, the random variables $D(\PRMspace_1),\dotsc, D(\PRMspace_i)$ are mutually independent random variables (cf. \cite{IkedaWatanabe}).

Fix $T\in(0,\infty)$ and let $(\Omega,\Fmc,\Pmb)$ be a complete probability space on which we are given a collection of i.i.d.\
Poisson random measures $\{D_i(ds\, dy\, dz)\}_{i\in\Nmb_0}$ on $[0,T]\times[0,1]\times\Rmb_+$ with intensity given by the Lebesgue measure.
Define the filtration $\{\hat{\Fmc}_t\}_{0\leq t\leq T}$ as
\begin{equation*}
	\hat{\Fmc}_t:=\sigma\{D_i((0,s]\times \PRMspace\times B),0\leq s\leq t, \PRMspace\in\Bmc([0,1]), B\in\Bmc(\Rmb_+)\}
\end{equation*}
and let $\{\Fmc_t\}_{0\leq t\leq T}$ be the $\Pmb$-augmentation of this filtration.
Using the above collection of PRM we now construct certain point processes with points in $[0,T]\times [0,1]$ as follows. 

Let $\bar{\Fmc}$ be the $\{\Fmc_t\}_{0\leq t\leq T}$-predictable $\sigma$-field on $\Omega\times[0,T]$.
Denote by $\bar{\controlset}_+$ the class of all $(\bar{\Fmc}\otimes \Bmc([0, 1]))/\Bmc(\Rmb_+)$-measurable
maps from $\Omega\times[0, T]\times [0, 1]$ to $\Rmb_+$.
For $\varphi\in\bar{\controlset}_+$ and each $i\in\Nmb_0$, define the counting process $D_i^\varphi$ on $[0, T]\times[0, 1]$ by
\begin{equation*}
	D_i^\varphi([0, t] \times \PRMspace)
	:=
	\int_{[0,t]\times \PRMspace}\one_{[0,\varphi(s,y))}(z) D_i(ds\, dy\, dz), \text{ for } t \in [0, T],\ \PRMspace \in \Bmc([0, 1]).
\end{equation*}
We regard $D_i^\varphi$ as a controlled random measure, where $\varphi$ is the control process that can be used to produce a
desired intensity. We will write $D_i^\varphi$ as $D^\theta_i$ if $\varphi = \theta$ for some constant $\theta\in\Rmb_+$.
In particular we will frequently take $\theta = n$.
Note that $D_i^\theta$ is a PRM on $[0, T] \times [0, 1]$ with intensity $\theta\, ds\, dy$.

For notational convenience, let $\Xt := [0,t]\times [0,1]$.
Also, for $x\in [0,1]$ with $nx \in \Nmb$, define
$\beta_n(x) := \binom{nx}{d_n} / \binom{n}{d_n}$ which 
equals the 
probability that when one samples $d_n$ random servers without replacement from the $n$ servers, the collection obtained is a subset of a given collection of $nx$ many servers. Extend the definition of $\beta_n$ to all of $[0,1]$ by setting
\begin{equation}
    \label{eq:beta}
    \beta_n(x) := \prod_{i=0}^{d_n-1} \left( \frac{x-\frac{i}{n}}{1-\frac{i}{n}} \right)^+, \;\; x \in [0,1].
\end{equation}
By using $D_0$ to represent the arrival process and $D_i$ to represent the departure process from queues with $i$ customers, $i \in \Nmb$, we can now give the state evolution of  $\Xbd^n$ as follows,
\begin{align}
	X^n_i(t) &= X^n_i(0) - \frac{1}{n} \int_{\Xt} \one_{[0,X^n_i(s-)-X^n_{i+1}(s-))}(y) D_i^{n}(ds\,dy) \notag \\
	&\qquad + \frac{1}{n} \int_{\Xt} \one_{[\beta_n(X^n_i(s-)), \beta_n(X^n_{i-1}(s-)))}(y) D_{0}^{n\lambda_n}(ds\,dy), \label{eq:state-process}
\end{align}
where $X^n_0(t) \equiv 1$ for all $t \in [0,T]$.
Note that $\beta_n(X_0^n(t)) \equiv 1$.
The first term on the right side equals the proportion of queues at time $0$ that are of length $i$ or more, the second term captures the number of departures from queues of length exactly $i$ during $[0,t]$ (note that any such departure only affects $X^n_i$ and keeps $X^n_j$ for $j\neq i$ unchanged) and the third term describes the number of arrivals to a queue with exactly $i-1$ jobs during $[0,t]$. 
Observe that $\beta_n(X^n_{i-1}(s-)))- \beta_n(X^n_i(s-))$ is the (conditional) probability that, given that there is an arrival at time $s$ to the dispatcher, it is routed to a queue with exactly $i-1$ jobs, thus the indicator in the third term corresponds to the JSQ$(d_n)$ policy described at the start of this section. 
Also note that when $d_n=n$, the above conditional probability degenerates to $\one_{\{X^n_{i-1}(s-) = 1, X^n_{i}(s-)<1\}}$ and thus matches with the term in the evolution of the JSQ system given in \cite{BudhirajaFriedlanderWu2021many} (see equations (2.1)-(2.2) therein).

Following \cite{BudhirajaFriedlanderWu2021many}, we rewrite the evolution of $X_i^n$ as follows:
\begin{equation}
	\label{eq:X}
	X_i^n(t) = Y_i^n(t) + \eta_{i-1}^n(t) - \eta_i^n(t), \quad i \ge 1,
\end{equation}
where $\eta_0^n(t) \equiv 0$ and
\begin{align}
	Y_1^n(t) & = X^n_1(0) + \frac{1}{n} \int_{\Xt} D_{0}^{n\lambda_n}(ds\,dy) - \frac{1}{n} \int_{\Xt} \one_{[0,X^n_1(s-)-X^n_2(s-))}(y) D_1^{n}(ds\,dy), \label{eq:Y1} \\
	Y_i^n(t) & = X^n_i(0) - \frac{1}{n} \int_{\Xt} \one_{[0,X^n_i(s-)-X^n_{i+1}(s-))}(y) D_i^{n}(ds\,dy), \quad i \ge 2, \label{eq:Yi} \\
	\eta_i^n(t) & =\frac{1}{n} \int_{\Xt} \one_{[0, \beta_n(X^n_i(s-)))}(y) D_{0}^{n\lambda_n}(ds\,dy), \quad i \ge 1. \label{eq:eta}
\end{align}

For ease of presenting the LDP, we make the following assumption throughout the paper, which in particular assumes deterministic initial states. Note that the first part of the assumption was noted previously.

\begin{Assumption}
	\label{assump:1}
	$d_n \to \infty$ and $\lambda_n \to \lambda \in (0,\infty)$ as $n \to \infty$. 
	There exist a sequence of $\xbd^n$ and $\xbd$ in $\bell_1^\downarrow$ such that $\Xbd^n(0)=\xbd^n$ and $\|\xbd^n - \xbd\|_1 \to 0$ as $n \to \infty$.
\end{Assumption}

For each $n \in \mathbb{N}$, $(\Xbd^n, \Ybd^n)$ is a $\pathspace$-valued random variable. 
This follows from the assumption on the initial condition and the fact that on any compact interval there can be at most finitely many jumps for the $\Ybd^n$ and $\Xbd^n$ processes a.s.\ and each jump is of the form $\pm n^{-1} \ebd_k$ where $\ebd_k$ is the element of $\bell_1$ with $1$ at the $k$-th coordinate and zeros elsewhere.
The main result of this work shows that as $n\to \infty$, the sequence $\{(\Xbd^n, \Ybd^n)\}_{n\in \mathbb{N}}$ satisfies a LDP in the above space.

\subsection{Rate Function}
\label{sec:rateFunction}

We first introduce the Skorokhod problem that will be used in the definition of the LDP rate function and summarize its properties.
For each $M \in \Nmb \cup \{\infty\}$, consider a (possibly infinite) matrix $R_M$ defined as
$$R_M(i,j) := -\one_{\{j=i\}} + \one_{\{j=i-1, i>1\}}, \mbox{ for } (i,j) \in \{1,2,\dotsc,M\}^2.$$
Let $\Vmb := (-\infty, 1]$.
Let $\Dmb_0([0,T]:\Rmb^M)$ be the subset of $\Dmb([0,T]:\Rmb^M)$ consisting of paths $\psibd$ such that $\psibd(0) \in \Vmb^M$.
\begin{Definition}
	\label{def:Smap}
	Let $M \in \Nmb \cup \{\infty\}$ and $\psibd \in \Dmb_0([0,T]:\Rmb^M)$. 
	Then $(\phibd,\etabd) \in \Dmb([0,T]:\Vmb^M \times \Rmb^M)$ is said to solve the Skorokhod problem for $\psibd$ associated with the reflection matrix $R_M$ if the following hold:
	\begin{enumerate}[(i)]
	\item
		$\phibd(t)=\psibd(t)+R_M\etabd(t)$ for all $t \in [0,T]$, namely 
		$$\phi_1(t)=\psi_1(t)-\eta_1(t),\;\; \phi_i(t)=\psi_i(t)+\eta_{i-1}(t)-\eta_i(t) \mbox{ for all } 2 \le i \le M \mbox{ and } t \in [0,T].$$
	\item 
		For each $i\in\{1,2,\dotsc,M\}$,
		$\eta_i(0)=0$, $\eta_i$ is nondecreasing, and $\int_0^T (1-\phi_i(s)) \,d\eta_i(s)=0$.
	\end{enumerate}
\end{Definition}

The structure of $R_M$ guarantees that there is always a unique solution $(\phibd,\etabd)$ to the Skorokhod problem for $\psibd \in \Dmb_0([0,T]:\Rmb^M)$; see  \cite[Lemma 2.2]{BudhirajaFriedlanderWu2021many}.
We denote the Skorokhod map $\Gamma_M \colon \Dmb_0([0,T]:\Rmb^M) \to \Dmb([0,T]:\Vmb^M)$ as $\Gamma_M(\psibd)=\phibd$ if $(\phibd,\etabd)$ solves the Skorokhod problem posed by $\psibd$.

\begin{Remark}
	\label{rmk:Skorokhod-map}
	For $M \in \Nmb \cup \{\infty\}$, it is easy to verify that if $\psibd \in \Dmb_0([0,T]:\Rmb^M)$ is such that $\psi_i$ is continuous (resp.\ absolutely continuous) for each $i$, and $\zetabd = \Gamma_M(\psibd)$, then $\zeta_i$ is continuous (resp. absolutely continuous) for each $i$.
	For $M \in \Nmb$, the structure of $R_M$ also guarantees that (cf.\ \cite[Proof of Lemma 2.2]{BudhirajaFriedlanderWu2021many} or \cite{HarrisonReiman1981reflected,DupuisIshii1991lipschitz})
	\begin{equation}
		\label{eq:Skorokhod-Lipschitz}
		\|\Gamma_M(\psibd)-\Gamma_M(\widetilde\psibd)\|_1 \le C_M\|\psibd-\widetilde\psibd\|_1
	\end{equation}
	for some $C_M \in (0,\infty)$.
\end{Remark}

Let $\Cmc$ be the subset of $\Cmb([0,T]:\bell_1^\downarrow \times \bell_1)$ consisting of all functions $(\zetabd,\psibd)$ such that
\begin{enumerate}[(i)]
\item
	$\zetabd(0)=\psibd(0)=\xbd$. $\zeta_i$ and $\psi_i$ are absolutely continuous on $[0,T]$ for each $i \in \Nmb$.
\item 
	$\zetabd = \Gamma_\infty(\psibd)$.
	That is, for some $\etabd = \{\eta_i, i \in \Nmb\} \in \Cmb([0,T]:\Rmb^{\infty})$,
	$(\zetabd, \etabd)$ solves the Skorokhod problem for $\psibd$ associated with $R_\infty$:
	\begin{equation}
		\label{eq:zeta_psi_eta}
		\zeta_i(t) = \psi_i(t) +\eta_{i-1}(t) - \eta_i(t), \; t \in [0,T],\; i\in\Nmb,
	\end{equation}
	where $\eta_0(t) \equiv 0$ and for every $i\ge 1$, $\eta_i(0)=0$, $\eta_i$ is non-decreasing, and $\int_0^T (1-\zeta_i(s)) \, \eta_i(ds) = 0$.
\end{enumerate}

\begin{Remark}
\label{rem:2.2}
In the setting of JSQ it is easy to see that $\Xbd^n = \Gamma_\infty(\Ybd^n)$ and an analogous property holds for the controlled analogues of these processes that arise in the large deviation analysis. 
This relation was important in the tightness proofs of \cite{BudhirajaFriedlanderWu2021many} (see e.g.\ the proof of Lemma 3.3 therein). 
In the setting of JSQ$(d_n)$ with $d_n <n$, an arrival to a queue of length $i$ may occur even when there are available servers with queue  lengths at most $i-1$, due to which the above identity fails to hold.  
This is one of the issues that requires a different approach in the analysis.
\end{Remark}

\begin{Remark}
\label{rem:2.3}
	From $\zetabd \in \Cmb([0,T]:\bell_1^\downarrow)$ we see that there exists a smallest $M = M(\zetabd) \in \Nmb$ such that $\sup_{t \in [0,T]} \zeta_M(t)<1.$
	Thus one only needs to consider an $M$-dimensional Skorokhod problem for $(\psi_i)_{i=1}^M$ (associated with $(\Vmb^M, R_M)$) in \eqref{eq:zeta_psi_eta}, although this $M$ will depend on the choice of $\zetabd$.
\end{Remark}

We now introduce the rate function that will govern the LDP.
Recall $\ell(z) = z\log(z)-z+1$ for $z\ge 0$, and let $\vartheta_0 := \lambda$ and $\vartheta_i := 1$ for $i\in\Nmb$.
For $(\zetabd,\psibd) \in \Cmc$, define
\begin{equation}
    \label{eqn:JSQRateFunction}
	\rate(\zetabd,\psibd) := \inf_{\varphibd\in\Smc(\zetabd,\psibd)}\left\{\sum_{i=0}^\infty\int_{\XT}\vartheta_i\ell(\varphi_i(s,y))\,ds\,dy\right\},
\end{equation}
where the set $\Smc(\zetabd,\psibd)$ consists of all $\varphibd= (\varphi_i)_{i \in \Nmb_0}$, where each $\varphi_i: [0,T] \times [0,1] \to \Rmb_+$ is such that
\begin{align}
	\psi_1(t) & = x_1 + \lambda \int_{\Xt}\varphi_0(s,y)\,ds\,dy  - \int_{\Xt}\one_{[0,\zeta_1(s)-\zeta_2(s))}(y) \varphi_{1}(s,y) \,ds\,dy, \label{eq:psi1}\\
	\psi_i(t) & = x_i - \int_{\Xt}\one_{[0,\zeta_i(s)-\zeta_{i+1}(s))}(y) \varphi_{i}(s,y) \,ds\,dy, \quad i \ge 2. \label{eq:psii}
\end{align}
For $(\zetabd,\psibd) \notin \Cmc$, define $\rate(\zetabd,\psibd) := \infty$.
Note that when $\varphi_i$ is taken to be $1$ for each $i$ in the above equations, $(\zetabd,\psibd)$ corresponds to the law of large numbers limit of the constrained and free processes $\{(\Xbd^n, \Ybd^n)\}_{n\in\Nmb}$(see \cite[Theorem 2.1]{BhamidiBudhirajaDewaskar2022near}).
Clearly, with this choice of $\{\varphi_i\}_{i\in\Nmb_0}$, the cost on the right side of \eqref{eqn:JSQRateFunction} is zero which verifies that the rate function evaluated at the LLN limit  is $0$.
For a general pair $(\zetabd,\psibd)$, the rate function is obtained by considering all controls $\{\varphi_i\}_{i\in\Nmb_0}$ that produce
the pair $(\zetabd,\psibd)$ through the system of equations in \eqref{eq:psi1}--\eqref{eq:psii} and by then taking infimum over the cost for all such controls as on the right side of \eqref{eqn:JSQRateFunction}.

\subsection{Main Result}
\label{sec:mainResult}

We begin by recalling the definition of a Large Deviation Principle.

\begin{Definition}
    Let $\Smb$ be a Polish space, $\{Z^n\}_{n \in \Nmb}$ be a sequence of $\Smb$-valued random variables, and $I$ be a function from $\Smb$ to $[0,\infty]$.
    We say that the sequence $\{Z^n\}_{n \in \Nmb}$ satisfies a large deviation principle on $\Smb$ with rate function $I$ and speed $n$ if the following three conditions hold:
    \begin{itemize}
        \item Large deviation upper bound: For each closed subset $F$ of $\Smb$,
        $$\limsup_{n \to \infty} \frac{1}{n} \log \Pmb(Z^n \in F) \le -\inf_{z \in F} I(z).$$
        \item Large deviation lower bound: For each open subset $G$ of $\Smb$,
        $$\liminf_{n \to \infty} \frac{1}{n} \log \Pmb(Z^n \in G) \ge -\inf_{z \in G} I(z).$$
        \item $I$ is a (good) rate function: For each $M \in [0,\infty)$, the level set $\{z \in \Smb : I(z) \le M\}$ is a compact subset of $\Smb$.
    \end{itemize}
\end{Definition}

We now present the main result of this work.

\begin{Theorem}
    \label{thm:mainResult}
	The function $\rate$ defined in \eqref{eqn:JSQRateFunction} is a rate function on $\pathspace$.
	The sequence $(\Xbd^n,\Ybd^n)$ satisfies a large deviation principle on $\pathspace$ with rate function $\rate$ and speed $n$.
\end{Theorem}

\begin{proof}
	From the equivalence between a LDP and a Laplace Principle (cf.\ \cite[Section 1.2]{dupuis2011weak} and \cite[Section 1.2]{BudhirajaDupuis2019analysis}), it suffices to establish the following three statements.
	\begin{enumerate}[(1)]
	\item
		Laplace Upper Bound: For all $G\in\Cmb_b(\pathspace)$, 
		\begin{align}\label{eqn:LaplaceUpperBound}
			\limsup_{n\to\infty}\frac{1}{n}\log\Emb e^{-nG(\Xbd^n,\Ybd^n)} \leq -\inf_{(\zetabd,\psibd)\in \Cmc}\{\rate(\zetabd,\psibd)+G(\zetabd,\psibd)\}.
		\end{align}
	\item
		Laplace Lower Bound: For all $G\in\Cmb_b(\pathspace)$,
		\begin{align}\label{eqn:LaplaceLowerBound}
			\liminf_{n\to\infty}\frac{1}{n}\log\Emb e^{-nG(\Xbd^n,\Ybd^n)}\geq -\inf_{(\zetabd,\psibd)\in\Cmc}\{\rate(\zetabd,\psibd)+G(\zetabd,\psibd)\}.
		\end{align}
	\item 
		$\rate$ is a rate function, namely for each $M\in[0,\infty),\{(\zetabd,\psibd)\in \Cmc:\rate(\zetabd,\psibd)\leq M\}$ is compact.
	\end{enumerate}
	Statements (1) and (2) are proved in Sections \ref{sec:UpperBound} and \ref{sec:LowerBound}, respectively.
    The proof of the third statement is given in Section \ref{sec:compactSets}. 
\end{proof}

The LDP in Theorem \ref{thm:mainResult} is useful in obtaining estimates for probabilities of various types of rare events in the JSQ$(d_n)$ system.
In particular, the formulation in $\pathspace$ (as opposed to $\pathspacer$) allows us to obtain estimates for probabilities of rare events involving quantities such as the total number of customers waiting in the system or total number of customers in queues of lengths $k$ or higher.
We illustrate the idea through the following example in the critical regime ($\lambda=1$) with all servers (asymptotically) busy ($x_1=1$). 
These two assumptions lead to some simplifications in the associated calculus of variations problem that we exploit.
Proofs will be given in Section \ref{sec:examples}.
 
\begin{Theorem}
	\label{thm:large-customers}
	Suppose $\lambda=1$ and $x_1=1$. 
	Fix $\varepsilon > 0$ and let
	\begin{align*}
		G^n_\varepsilon & := \{\|\Xbd^n(t)\|_1 > \|\xbd^n\|_1 + \varepsilon \mbox{ for some } t \in [0,T]\}, \\ 
		F^n_\varepsilon & := \{\|\Xbd^n(t)\|_1 \ge \|\xbd^n\|_1 + \varepsilon \mbox{ for some } t \in [0,T]\}.
	\end{align*}
	Then 
	\begin{equation*}
		\lim_{n\to\infty}\frac{1}{n}\log\Pmb(G^n_\varepsilon) =
		\lim_{n\to\infty}\frac{1}{n}\log\Pmb(F^n_\varepsilon) = - T\ell\left( \frac{\frac{\varepsilon}{T}+\sqrt{4+(\frac{\varepsilon}{T})^2}}{2} \right) - T\ell\left( \frac{-\frac{\varepsilon}{T}+\sqrt{4+(\frac{\varepsilon}{T})^2}}{2} \right)
	\end{equation*}
	and
	\begin{equation*}
		\lim_{T \to \infty}\lim_{n\to\infty}\frac{T}{n}\log(\Pmb(G^n_\varepsilon)) = \lim_{T \to \infty}\lim_{n\to\infty}\frac{T}{n}\log(\Pmb(F^n_\varepsilon)) = -\frac{\varepsilon^2}{4}.
	\end{equation*}
\end{Theorem}

As an immediate corollary, Theorem \ref{thm:large-customers} gives large deviation estimates for probabilities of rare events involving long queues in the system.

\begin{Corollary}
    \label{cor:long-queue}
    Suppose $\lambda=1$, $x^n_1 = 1 = x_1$ and $x^n_i = 0 = x_i$ for $i\geq 2$ and $n \in \Nmb$.
	Fix $j \ge 3$ and let
	\begin{equation*}
		U^n_j := \{X_j^n(t) > 0 \mbox{ for some } t \in [0,T]\}, \quad V^n_j := \{X_{j-1}^n(t) = 1 \mbox{ for some } t \in [0,T]\}.
	\end{equation*}
    \begin{enumerate}[(a)]
    \item 
        Then
        \begin{equation}\label{eq:eq332r}
    	\begin{aligned}
    		\liminf_{n\to\infty}\frac{1}{n}\log\Pmb(U^n_j) & \ge - T\ell\left( \frac{\frac{j-2}{T}+\sqrt{4+(\frac{j-2}{T})^2}}{2} \right) - T\ell\left( \frac{-\frac{j-2}{T}+\sqrt{4+(\frac{j-2}{T})^2}}{2} \right), \\
    		\limsup_{n\to\infty}\frac{1}{n}\log\Pmb(V^n_j)
    		& \le - T\ell\left( \frac{\frac{j-2}{T}+\sqrt{4+(\frac{j-2}{T})^2}}{2} \right) - T\ell\left( \frac{-\frac{j-2}{T}+\sqrt{4+(\frac{j-2}{T})^2}}{2} \right).
    	\end{aligned}  
        \end{equation}
    \item 
        Suppose $d_n=n$. Then the above inequalities are equalities, namely,
    	\begin{align*}
    		\lim_{n\to\infty}\frac{1}{n}\log\Pmb(U^n_j) & =
    		\lim_{n\to\infty}\frac{1}{n}\log\Pmb(V^n_j) \\
    		& = - T\ell\left( \frac{\frac{j-2}{T}+\sqrt{4+(\frac{j-2}{T})^2}}{2} \right) - T\ell\left( \frac{-\frac{j-2}{T}+\sqrt{4+(\frac{j-2}{T})^2}}{2} \right)
    	\end{align*}
    	and
    	\begin{equation*}
    		\lim_{T \to \infty}\lim_{n\to\infty}\frac{T}{n}\log(\Pmb(U^n_j)) = \lim_{T \to \infty}\lim_{n\to\infty}\frac{T}{n}\log(\Pmb(V^n_j)) = -\frac{(j-2)^2}{4}.
    	\end{equation*}       
    \end{enumerate}
\end{Corollary}

\begin{proof}
    (a)
        The result follows from Theorem \ref{thm:large-customers} on noting that $U_j^n \supset G_{j-2}^n$ and $V_j^n \subset F_{j-2}^n$.
        
    (b) 
        The result follows from Theorem \ref{thm:large-customers} on noting that $U_j^n = G_{j-2}^n$ and $V_j^n = F_{j-2}^n$ when $d_n=n$ and $\xbd^n=(1,0,0,\dotsc)$.
\end{proof}

\begin{Remark}
    Corollary \ref{cor:long-queue}(b) was proved in \cite[Theorem 2.5]{BudhirajaFriedlanderWu2021many} by solving the associated calculus of variation problem.
    Here it follows as an immediate consequence of the general result in Theorem \ref{thm:large-customers}.
\end{Remark}

\begin{Remark}
\label{rem:diffdec}
One may wonder whether the inequalities \eqref{eq:eq332r} can be replaced by equalities, namely $\Pmb(U^n_j)$ and $\Pmb(V^n_j)$ for general $d_n$ have the   same asymptotic behavior as in the case $d_n=n$.
Note that the relation $U_j^n \subset V_j^n$, which holds when $d_n=n$, fails for general $d_n \to \infty$. However, one would still be able to replace the inequalities in \eqref{eq:eq332r} with equalities if the event $U_j^n \setminus V_j^n$ was exponentially negligible.
Unfortunately, this is not true in general.
Consider for example, $j=3$. We will show that
\begin{equation}\label{eq:300}\lim_{n \to \infty} \frac{1}{n} \log \Pmb(U^n_3) = 0\end{equation}
as long as $d_n = o(n)$ and consequently in this case the first inequality in \eqref{eq:eq332r} is strict.
Denote by $A^n$ the event that there are $d_n+1$ arrivals before the first departure (so that after the first $d_n$ arrivals we will have $d_n$ queues of length $2$ and remaining $n-d_n$ queues of length $1$); the $(d_n+1)$-th arrival goes to some queue with length $2$ (namely those $d_n$ queues with length $2$ are chosen); and all these $d_n+1$ jumps occur before time $T$.
Then we have
$$\Pmb(U^n_3) \ge \Pmb(A^n) = \left(\frac{1}{2}\right)^{d_n} \cdot \frac{1}{2} \frac{1}{\binom{n}{d_n}} \cdot c_n.$$
Here $\left(\frac{1}{2}\right)^{d_n}$ is the probability that the first $d_n$ arrivals occur before the first departure, $\frac{1}{2} \frac{1}{\binom{n}{d_n}}$ is the probability that the $(d_n+1)$-th arrival occurs before the first departure and goes to some queue with length $2$, and $c_n := \Pmb(\textnormal{Gamma}(d_n+1,2n) \le T)$ is the probability that these happen before time $T$.
Then
\begin{equation*}
    \frac{1}{n} \log \Pmb(U^n_3) \ge \frac{d_n+1}{n} \log \frac{1}{2} + \frac{1}{n} \log \frac{1}{\binom{n}{d_n}} + \frac{1}{n} \log c_n.
\end{equation*}
Since $d_n=o(n)$, we have $\frac{d_n+1}{n} \log \frac{1}{2} \to 0$ and $\frac{1}{n} \log c_n \to 0$ as $n \to \infty$.
As for the middle term, by Stirling's formula, we have
\begin{align*}
    \lim_{n \to \infty} \frac{1}{n} \log \frac{1}{\binom{n}{d_n}} & = \lim_{n \to \infty} \frac{1}{n} \log \frac{(n-d_n)! d_n !}{n!} = \lim_{n \to \infty} \frac{1}{n} \log \frac{\left(\frac{n-d_n}{e}\right)^{n-d_n}\sqrt{n-d_n} \left(\frac{d_n}{e}\right)^{d_n}\sqrt{d_n}}{\left(\frac{n}{e}\right)^{n}\sqrt{n}} \\
    & = \lim_{n \to \infty} \frac{n-d_n}{n} \log \frac{n-d_n}{n} + \lim_{n \to \infty} \frac{d_n}{n} \log \frac{d_n}{n} = 0.
\end{align*}
Therefore $$\liminf_{n \to \infty} \frac{1}{n} \log \Pmb(U^n_3) \ge 0,$$ which gives the claimed equality in \eqref{eq:300}.

In fact, when $d_n=o(n)$, similar arguments as above show that
$$\liminf_{n \to \infty} \frac{1}{\theta_n} \log \Pmb(U^n_3) \ge 0$$
for all $\theta_n \gg d_n |\log \frac{d_n}{n}|$.
For example, if $d_n=\sqrt{n}$, then $\theta_n \gg \sqrt{n}\log n$ suffices.

We leave as open the question whether the inequalities in \eqref{eq:eq332r} can be replaced by equalities when $d_n=\Theta(n)$.
\end{Remark}

\section{Representation and Weak Convergence of Controlled Processes}\label{sec:varWeakCon}
In this section we give several preparatory results that are needed for the proofs of both the upper and the lower bounds (i.e.  \eqref{eqn:LaplaceUpperBound} and \eqref{eqn:LaplaceLowerBound}).
Section \ref{sec:varRep} presents a variational representation from \cite{budhiraja2011variational} (see also \cite[Theorem 8.12]{BudhirajaDupuis2019analysis}) that is the starting point of our analysis.
In Section \ref{sec:tight} we prove tightness of certain families of controls and controlled processes which arise from the variational representation of Section \ref{sec:varRep}.
Finally, Section \ref{sec:weakconv} presents a result which characterizes the distributional limit points of this collection of processes.

\subsection{Variational Representation}\label{sec:varRep}
Recall that $\bar{\controlset}_+$ denotes the class of $(\bar{\Fmc}\otimes \Bmc([0, 1]))/\Bmc(\Rmb_+)$-measurable
maps from $\Omega\times[0, T]\times [0, 1]$ to $\Rmb_+$.
For each $m\in\Nmb$ let
\begin{align*}
	\bar{\controlset}_{b,m}
	&:= \{(\varphi_i)_{i\in\Nmb_0}:\varphi_i\in\bar{\controlset}_+\text{ for all }i\in\Nmb_0, \text{  for all } (\omega,t,y)\in\Omega\times[0,T]\times[0,1]\\
	&\qquad \frac{1}{m}\leq \varphi_i(\omega,t,y)\leq m \text{ for }i\leq m\text{ and } \varphi_i(\omega,t,y)=1 \text{ for } i>m\}
\end{align*}
and let $\bar{\controlset}_b:=\cup_{m=1}^\infty\bar{\controlset}_{b,m}$.
For each $n \in \Nmb$ and any $\varphibd^n\in\bar{\controlset}_b$ we denote by $(\bar\Xbd^{n,\varphibd^n},\bar\Ybd^{n,\varphibd^n},\bar\etabd^{n,\varphibd^n})$ the controlled analogues of $(\Xbd^n,\Ybd^n,\etabd^n)$ obtained by replacing the PRMs in \eqref{eq:Y1}--\eqref{eq:eta} with controlled point processes, $D_0^{n\lambda_n\varphi^n_0}$ and $D_i^{n\varphi^n_i}, i\in\Nmb$.
Namely, the state evolution equations for the controlled processes are as follows,
\begin{equation}
	\label{eq:controlled-process}
	\begin{aligned}
		\Xbar_i^{n,\varphibd^n}(t) & = \Ybar_i^{n,\varphibd^n}(t) + \etabar_{i-1}^{n,\varphibd^n}(t) - \etabar_i^{n,\varphibd^n}(t), \quad i \ge 1,\\
		\Ybar_1^{n,\varphibd^n}(t) & = x_1^n + \frac{1}{n} \int_{\Xt} D_{0}^{n\lambda_n\varphi^n_0}(ds\,dy) - \frac{1}{n} \int_{\Xt} \one_{[0,\Xbar^{n,\varphibd^n}_1(s-)-\Xbar^{n,\varphibd^n}_2(s-))}(y) D_1^{n\varphi^n_1}(ds\,dy), \\
		\Ybar_i^{n,\varphibd^n}(t) & = x_i^n - \frac{1}{n} \int_{\Xt} \one_{[0,\Xbar^{n,\varphibd^n}_i(s-)-\Xbar^{n,\varphibd^n}_{i+1}(s-))}(y) D_i^{n\varphi^n_i}(ds\,dy), \quad i \ge 2, \\
		\etabar_i^{n,\varphibd^n}(t) & =\frac{1}{n} \int_{\Xt} \one_{[0, \beta_n(\Xbar^{n,\varphibd^n}_i(s-)))}(y) D_{0}^{n\lambda_n\varphi^n_0}(ds\,dy), \quad i \ge 1,
	\end{aligned} 
\end{equation}
where $\Xbar_0^{n,\varphibd^n}(t) \equiv 1$ and $\etabar_0^{n,\varphibd^n} \equiv 0$ for all $t \in [0,T]$.
When it is clear from context which controls are being used we may simply write $(\bar\Xbd^{n},\bar\Ybd^{n},\bar\etabd^{n})$ to represent the controlled processes.

	Let $\vartheta^n_0 :=\lambda_n$ and $\vartheta^n_i := 1$ for $i\in\Nmb$.
	The following variational representation will be instrumental in proving both the upper and the lower bounds, namely \eqref{eqn:LaplaceUpperBound} and \eqref{eqn:LaplaceLowerBound}.
	For a proof we refer the reader to  \cite[Theorem 2.1]{budhiraja2011variational}, 
    \cite[Theorem 8.2]{BudhirajaDupuis2019analysis} and comments above \cite[Lemma 3.1]{BudhirajaFriedlanderWu2021many}. 
    
\begin{Lemma}
    \label{lem:varRep}
	Let $G\in\Cmb_b(\pathspace)$. Then
	\begin{align}\label{eqn:variationalRep}
		\begin{split}
		-\frac{1}{n} \log \Emb e^{-nG(\Xbd^n,\Ybd^n)} &= \inf_{\varphibd^n\in\bar{\controlset}_b} \Emb \left\{\sum_{i=0}^\infty \int_{\XT}\vartheta^n_i \ell(\varphi_i^n(s,y))\,ds\,dy + G(\bar\Xbd^n,\bar\Ybd^n) \right\}.
		\end{split}
	\end{align}
\end{Lemma}

\subsection{Tightness}\label{sec:tight}
In this section we prove a key tightness result which says that if the costs are appropriately bounded then the corresponding collection of controls and controlled processes is tight.
We begin by describing the topology on the space of controls.
For $M \in (0,\infty)$, denote by $S_M$ the collection of all  $\hbd= \{h_i\}_{i\in \Nmb_0}$, where $h_i: [0,T]\times [0,1] \to \Rmb_+$ for each $i \in \Nmb_0$ and 
$$
\sum_{i=0}^\infty \int_{\XT} \ell(h_i(s,y))\,ds\,dy\leq M.$$
Any $h_i$ as above can be identified with a finite measure $\nu^{h_i}$ on $[0,T]\times [0,1]$ by the following relation
$$\nu^{h_i}(\PRMspace) := \int_{\PRMspace} h_i(s,y)\,ds\,dy, \; \PRMspace \subset \Bmc([0,T]\times [0,1]).$$
The space $\Mmb$ of finite measures on $[0,T]\times[0,1]$ is equipped with the weak convergence topology and the space $\Mmb^{\infty}$ is equipped with the corresponding product topology.
Using the above identification, each element in $S_M$ can be mapped to an element of the Polish space $\Mmb^{\infty}$ and the space $S_M$ with the inherited topology is compact (see \cite[Lemma A.1]{budhiraja2013large}).

We record the following elementary lemma for future use. Proof is omitted.
\begin{Lemma}
	\label{lem:ellProp}
	Let $\ell(x) = x\log(x)-x+1$. Then the following properties hold for $\ell(x)$:
	\begin{enumerate}[(a)]
	\item
		For each $K>0$, there exists $\gamma(K)\in(0,\infty)$ such that $\gamma(K)\to0$ as $K\to\infty$ and $x\leq \gamma(K)\ell(x)$, for $x\geq K$.
	\item
		For $x\geq0$, $x\leq\ell(x)+2$.
	\end{enumerate}
\end{Lemma}

The following is the main tightness result of this section.

\begin{Lemma}\label{lem:tightness}
	Suppose that $\{\varphibd^n\}$ is a sequence in $\bar\controlset_b$ such that for some $M_0 \in (0,\infty)$
	\begin{equation}
		 \sup_{n\in\Nmb}\sum_{i=0}^\infty\int_{\XT} \ell(\varphi^n_i(s,y))\,ds\,dy\leq M_0  \mbox{ a.s. }\label{eqn:controlBound}
	\end{equation}
	Denote by $(\bar\Xbd^n, \bar\Ybd^n, \bar\etabd^n)$ the controlled processes associated with $\varphibd^n$, given by \eqref{eq:controlled-process}.
	Then, regarding $\varphibd^n$ as an $S_{M_0}$-valued random variable, the sequence
	$\{(\bar\Xbd^n, \bar\Ybd^n, \bar\etabd^n, \varphibd^n)\}_{n\in\Nmb}$ is tight in $\Dmb([0,T]:\bell_1^\downarrow \times \bell_1 \times \Rmb^\infty) \times S_{M_0}$. 
	Furthermore the collection $\{(\bar\Xbd^n, \bar\Ybd^n, \bar\etabd^n)\}_{n\in\Nmb}$ is $\Cmb$-tight. 
\end{Lemma}

\begin{proof}
	Since $S_{M_0}$ is compact the tightness of $\{\varphibd^n\}_{n \in \Nmb}$ is immediate.
	Recall $d_\infty$ defined in \eqref{eq:d-infty}.
	Noting that jump sizes of $\bar\Xbd^n$, $\bar\Ybd^n$, and $\bar\etabd^n$ (with respect to $\|\cdot\|_1$, $\|\cdot\|_1$, and $d_\infty$ respectively) are bounded by $1/n$, $\Cmb$-tightness follows once we have tightness of $\{(\bar\Xbd^n, \bar\Ybd^n, \bar\etabd^n)\}_{n\in\Nmb}$.
    By appealing to Aldous' tightness criteria (cf.\ \cite[Theorem 2.2.2]{joffe1986weak}), it suffices to show that
    \begin{equation}
        \mbox{for each } t \in [0,T], \mbox{ the sequence } \{(\bar\Xbd^n(t),\bar\Ybd^n(t),\bar\etabd^n(t))\}_{n\in\Nmb} \mbox{ is tight in } \bell_1^\downarrow \times \bell_1 \times \Rmb^\infty, \label{eq:tight-marginal} 
    \end{equation}
    and
    \begin{equation}
        \lim_{\delta\to0} \limsup_{n\to\infty} \sup_{\tau\in\Tmc^\delta} \Emb[\|\bar{\Xbd}^n(\tau+\delta)-\bar{\Xbd}^n(\tau)\|_1 + \|\bar{\Ybd}^n(\tau+\delta)-\bar{\Ybd}^n(\tau)\|_1 + d_\infty(\bar\etabd^n(\tau+\delta),\bar\etabd^n(\tau))]=0, \label{eq:tight-fluctuation}
    \end{equation}
    where $\Tmc^\delta$ is the set of all $[0,T-\delta]$-valued stopping times.

    We first prove \eqref{eq:tight-marginal}.
    Fix $t \in [0,T]$.
    For this, it suffices to show that $\{(\Xbar_i^n(t),\Ybar_i^n(t),\etabar_i^n(t))\}_{n\in\Nmb}$ is tight in $\Rmb^3$ for each $i \ge 1$, and that
    \begin{equation}
        \label{eq:tight-tail}
        \lim_{k \to \infty} \limsup_{n \to \infty} \Emb \sum_{i=k}^\infty [|\Xbar_i^n(t)| + |\Ybar_i^n(t)|] = 0.
    \end{equation}
    Now fix $i \ge 1$. 
    From \eqref{eq:controlled-process} we have
	\begin{align*}
		\Emb[|\Xbar_i^n(t)| + |\Ybar_i^n(t)| + |\etabar_i^n(t)|] 
        & \le \Emb[2|\Ybar_i^n(t)| + |\etabar_i^n(t)| + |\etabar_{i-1}^n(t) - \etabar_i^n(t)|] \\
		& \leq 2x_i^n + \Emb\int_{\Xt}[2\varphi^n_{i}(s,y)+4\lambda_n\varphi^n_{0}(s,y)]\,ds\,dy \\
		& \leq 2 + \Emb\int_{\Xt}[2(\ell(\varphi^n_{i}(s,y))+2)+4\lambda_n(\ell(\varphi^n_{0}(s,y))+2)]\,ds\,dy \\
        & \le 2 + (2+4\lambda_n)M_0 + 2(2+4\lambda_n)T,
	\end{align*}
	where the third inequality uses Lemma \ref{lem:ellProp}(b) and the last inequality uses \eqref{eqn:controlBound}. 
    Since $\sup_n \lambda_n < \infty$, we have tightness of $\{(\Xbar_i^n(t),\Ybar_i^n(t),\etabar_i^n(t))\}_{n\in\Nmb}$ in $\Rmb^3$. 
    Again from \eqref{eq:controlled-process} we have that for $k \ge 2$,
	\begin{align}
		\Emb \sum_{i=k}^\infty [|\Xbar_i^n(t)| + |\Ybar_i^n(t)|]
        & \le \Emb \sum_{i=k}^\infty [2|\Ybar_i^n(t)| + |\etabar_{i-1}^n(t) - \etabar_i^n(t)|] \notag \\
		& \leq \sum_{i=k}^\infty 2x_i^n + \Emb \sum_{i=k}^\infty \int_{\Xt} 2\varphi^n_{i}(s,y) \one_{[0,\Xbar^n_i(s)-\Xbar^n_{i+1}(s))}(y)\,ds\,dy \notag \\
        & \qquad + \Emb \int_{\Xt} \lambda_n\varphi^n_{0}(s,y) \one_{[0, \beta_n(\Xbar^{n}_{k-1}(s)))}(y) \,ds\,dy. \label{eq:tight-tail-0}
	\end{align}
    For the first term, from Assumption \ref{assump:1} we have
    \begin{equation}
        \label{eq:tight-tail-1}
        \lim_{k \to \infty} \limsup_{n \to \infty} \sum_{i=k}^\infty x_i^n \le \lim_{k \to \infty} \sum_{i=k}^\infty x_i =0.
    \end{equation}
    For the second term in \eqref{eq:tight-tail-0}, first note that by non-negativity of $\bar\Xbd^n(t)$, \eqref{eq:controlled-process}, Lemma \ref{lem:ellProp}(b) and \eqref{eqn:controlBound},
    \begin{equation*}
        \Emb \|\bar\Xbd^n(t)\|_1 = \Emb \sum_{i=1}^{\infty} \bar Y^n_i(t)\le \|\xbd^n\|_1 + \Emb \int_{\Xt} \lambda_n \varphi^n_0(s,y)\,dsdy \le \|\xbd^n\|_1 + \lambda_n(M_0+2T),
    \end{equation*}
    and hence
    \begin{equation}
        \label{eq:tight-X-bound}
        \sup_n \sup_{t \in [0,T]} \Emb \|\bar\Xbd^n(t)\|_1 < \infty.
    \end{equation}
    Therefore, for any $K > 0$, we have
    \begin{align*}
        & \lim_{k \to \infty} \limsup_{n \to \infty} \Emb \sum_{i=k}^\infty \int_{\Xt} \varphi^n_{i}(s,y) \one_{[0,\Xbar^n_i(s)-\Xbar^n_{i+1}(s))}(y) \,ds\,dy \\
        & \le \lim_{k \to \infty} \limsup_{n \to \infty} \Emb \sum_{i=k}^\infty \int_{\Xt} [K + \gamma(K) \ell(\varphi^n_{i}(s,y))] \one_{[0,\Xbar^n_i(s)-\Xbar^n_{i+1}(s))}(y) \,ds\,dy \\
        & \le \lim_{k \to \infty} \limsup_{n \to \infty} K \Emb \int_0^t \Xbar^n_k(s)\,ds + \gamma(K) M_0 \\
        & \le \lim_{k \to \infty} \limsup_{n \to \infty} K \Emb \int_0^t \frac{\|\bar\Xbd^n(s)\|_1}{k} \,ds + \gamma(K) M_0 \\
        & = \gamma(K) M_0 
    \end{align*}
    which converges to $0$ as $K \to \infty$.
    Here the second line uses Lemma \ref{lem:ellProp}(a), the third uses \eqref{eqn:controlBound}, the fourth uses the monotonicity of $j \mapsto \bar X^n_j(s)$, and the last line uses \eqref{eq:tight-X-bound}.
    Similarly, for the third term in \eqref{eq:tight-tail-0} and any $K>0$,
    \begin{align*}
        & \lim_{k \to \infty} \limsup_{n \to \infty} \Emb \int_{\Xt} \lambda_n\varphi^n_{0}(s,y) \one_{[0, \beta_n(\Xbar^{n}_{k-1}(s)))}(y) \,ds\,dy \\
        & \le \lim_{k \to \infty} \limsup_{n \to \infty} \Emb \int_{\Xt} \lambda_n[K+\gamma(K)\ell(\varphi^n_{0}(s,y))] \one_{[0, \beta_n(\Xbar^{n}_{k-1}(s)))}(y) \,ds\,dy \\
        & \le \lim_{k \to \infty} \limsup_{n \to \infty} \lambda_n K \Emb \int_0^t \Xbar^n_{k-1}(s)\,ds + \lambda \gamma(K) M_0 \\
        & = \lambda \gamma(K) M_0 
    \end{align*}
    which converges to $0$ as $K \to \infty$.
    Here the third line follows since $\beta_n(x) \le x^{d_n} \le x$ for $x \in [0,1]$. 
    Combining above two estimates with \eqref{eq:tight-tail-0} and \eqref{eq:tight-tail-1}, we get \eqref{eq:tight-tail}.
    This gives \eqref{eq:tight-marginal}. 

    Finally we prove \eqref{eq:tight-fluctuation}.    
	Fix $\delta \in (0,1)$, $\tau \in \Tmc^\delta$, and $K>0$. 
	From \eqref{eq:controlled-process} we have
	\begin{align*}
		&\Emb[\|\bar\Xbd^n(\tau+\delta)-\bar\Xbd^n(\tau)\|_1 + \|\bar\Ybd^n(\tau+\delta)-\bar\Ybd^n(\tau)\|_1 + d_\infty(\bar\etabd^n(\tau+\delta),\bar\etabd^n(\tau))] \\
		& \le 2\Emb \|\bar\Ybd^n(\tau+\delta)-\bar\Ybd^n(\tau)\|_1 + \Emb \sum_{i=1}^\infty |[\etabar_{i-1}^n(\tau+\delta)-\etabar_{i}^n(\tau+\delta)] - [\etabar_{i-1}^n(\tau)-\etabar_{i}^n(\tau)]| \\
		& \qquad + \Emb \sum_{i=1}^\infty \frac{|\etabar_{i}^n(\tau+\delta) -\etabar_{i}^n(\tau)|}{2^i}\\
		& \le 2\Emb \sum_{i=1}^\infty \int_{[\tau,\tau+\delta]\times[0,1]} \varphi^n_{i}(s,y) \one_{[0,\Xbar^n_i(s)-\Xbar^n_{i+1}(s))}(y) \,ds\,dy + 4\lambda_n \Emb \int_{[\tau,\tau+\delta]\times[0,1]} \varphi^n_{0}(s,y)\,ds\,dy
		\\
		& \le 2\Emb \sum_{i=1}^\infty \int_{[\tau,\tau+\delta]\times[0,1]} [K+\gamma(K)\ell(\varphi^n_{i}(s,y))] \one_{[0,\Xbar^n_i(s)-\Xbar^n_{i+1}(s))}(y) \,ds\,dy \\
		& \qquad + 4\lambda_n \Emb \int_{[\tau,\tau+\delta]\times[0,1]} [K+\gamma(K)\ell(\varphi^n_{0}(s,y))]\,ds\,dy \\
		& \le (2+4\lambda_n)\delta K + (2+4\lambda_n)\gamma(K)M_0,
	\end{align*}
	where the third inequality uses Lemma \ref{lem:ellProp}(a) and the last inequality uses \eqref{eqn:controlBound}.
	Therefore
	\begin{align*}
		\limsup_{\delta\to0} \limsup_{n\to\infty} \sup_{\tau\in\Tmc^\delta} & \Emb[\|\bar\Xbd^n(\tau+\delta)-\bar\Xbd^n(\tau)\|_1 + \|\bar\Ybd^n(\tau+\delta)-\bar\Ybd^n(\tau)\|_1 \\
        & \qquad \qquad
        + d_\infty(\bar\etabd^n(\tau+\delta),\bar\etabd^n(\tau))] \leq (2+4\lambda_n)\gamma(K)M_0,
	\end{align*}
	which goes to $0$ as $K \to \infty$.
	This gives \eqref{eq:tight-fluctuation} and completes the proof.
\end{proof}

\subsection{Characterization of Limit Points}
\label{sec:weakconv}

Suppose that $\{\varphibd^n\}_{n\in\Nmb}$ is a sequence as in Lemma \ref{lem:tightness}. 
Then from the lemma we have the tightness of $\{(\bar\Xbd^n, \bar\Ybd^n, \bar\etabd^n, \varphibd^n)\}_{n\in\Nmb}$. 
In this section we characterize the limit points of this sequence.
It will be convenient to consider the following compensated point processes
\begin{equation*}
	\tilde{D}^{n \vartheta_i^n\varphi_i^n}_i(ds\,dy)
	:= D^{n \vartheta_i^n\varphi_i^n}_i(ds\,dy) - n\vartheta_i^n\varphi_i^n(s,y)\,ds\,dy, \quad n \in \Nmb, \quad i \ge 0.
\end{equation*}
Define compensated processes $\widetilde\Bbd^n$ and $\widetilde\etabd^n$ as
\begin{align}
	\Btil^n_1(t) &:= \frac{1}{n} \int_{\Xt} \Dtil_{0}^{n\lambda_n\varphi_0^n}(ds\,dy) - \frac{1}{n} \int_{\Xt} \one_{[0,\Xbar^{n}_1(s-)-\Xbar^{n}_2(s-))}(y) \Dtil_1^{n\varphi_1^n}(ds\,dy), \label{eq:Btil-1} \\
	\Btil^n_i(t) &:= \frac{1}{n} \int_{\Xt} \one_{[0,\Xbar^{n}_i(s-)-\Xbar^{n}_{i+1}(s-))}(y) \Dtil_i^{n\varphi_i^n}(ds\,dy), \qquad i\geq 2, \label{eq:Btil-i} \\
	\etatil_i^n(t) &:= \frac{1}{n} \int_{\Xt} \one_{[0, \beta_n(\Xbar^{n}_i(s-)))}(y) \Dtil_{0}^{n\lambda_n\varphi_0^n}(ds\,dy), \qquad i \ge 1. \label{eq:etatil}
\end{align}
These allow us to write
\begin{align}
	\Ybar^n_1(t) &= x_1^n + \Btil^n_1(t) + \lambda_n \int_{\Xt} \varphi_0^n(ds\,dy) - \int_{\Xt} \one_{[0,\bar{X}^n_1(s)-\bar{X}^n_2(s))}(y) \varphi^n_1(s,y)\,ds\,dy, \label{eq:Ybar-decom-1} \\
	\Ybar^n_i(t) &= x_i^n - \Btil^n_i(t) - \int_{\Xt} \one_{[0,\bar{X}^n_i(s)-\bar{X}^n_{i+1}(s))}(y) \varphi^n_i(s,y)\,ds\,dy, \qquad i \ge 2. \label{eq:Ybar-decom-i}
\end{align}
The following lemma characterizes the limit points of $\{(\bar\Xbd^n, \bar\Ybd^n, \bar\etabd^n, \varphibd^n)\}_{n\in\Nmb}$.

\begin{Lemma}\label{lem:convergence}
	Suppose that $\{\varphibd^n\}$ is a sequence as in Lemma \ref{lem:tightness}.
	Suppose also that the associated sequence $\{(\bar\Xbd^n, \bar\Ybd^n, \bar\etabd^n, \varphibd^n)\}_{n\in\Nmb}$ converges along a subsequence, in distribution, to $(\bar\Xbd,\bar\Ybd,\bar\etabd,\varphibd)$ given on some probability space $(\Omega^*,\Fmc^*,\Pmb^*)$.
	Then the following holds $\Pmb^*$-a.s.
	\begin{enumerate}[(a)]
	\item
		Equations \eqref{eq:psi1}--\eqref{eq:psii} are satisfied with $(\zetabd, \psibd, \varphibd)$ replaced by $(\bar\Xbd, \bar\Ybd, \varphibd)$.
	\item
		$(\bar\Xbd, \bar\Ybd)\in \Cmc$ and $\varphibd\in\Smc(\bar\Xbd, \bar\Ybd)$.
		In particular, $(\bar\Xbd, \bar\Ybd, \bar\etabd)$ satisfy the following system of equations
		\begin{align}
			\bar{X}_1(t) &= \Ybar_1(t) - \etabar_1(t), \label{eq:XbarLim1} \\
			\bar{X}_i(t) &= \Ybar_i(t) + \etabar_{i-1}(t) - \etabar_{i}(t), \quad i \ge 2, \label{eq:XbarLim2}
		\end{align}
		and for every $i \in \Nmb$,  $\etabar_i(0)=0$, $\etabar_i$ is non-decreasing, and $\int_0^t (1-\Xbar_i(s)) \, \etabar_i(ds) = 0$.
	\end{enumerate}
\end{Lemma}

\begin{proof}
	Assume without loss of generality that convergence occurs along the whole sequence.
	Recall the notations in \eqref{eq:Btil-1}--\eqref{eq:Ybar-decom-i}.
	It follows from Doob's inequality and Lemma \ref{lem:ellProp}(b) that for each $i \ge 1$,
	\begin{align}
		& \Emb \left( \sup_{0\leq t\leq T}|\ti B^n_i(t)|^2 + \sup_{0\leq t\leq T}|\etatil^n_i(t)|^2 \right) \notag \\
		& \leq \frac{1}{n}\Emb\int_{\XT}[12\lambda_n\varphi^n_{0}(s,y)+ 8\varphi^n_{i}(s,y)]\,ds\,dy \nonumber\\
		& \leq \frac{1}{n}\Emb\int_{\XT}[12\lambda_n (\ell(\varphi^n_{0}(s,y))+2)+ 8(\ell(\varphi^n_{i}(s,y))+2)]\,ds\,dy \nonumber\\
		& \leq \frac{1}{n}(12\lambda_n+8)(M_0+2T) \to 0 \label{eq:cvg-0}
	\end{align}
	as $n \to \infty$.
	By appealing to the Skorokhod representation theorem {(cf.\ \cite[Theorem 6.7]{BillingsleyConv})}, we can assume without loss of generality that $(\bar\Xbd^n, \bar\Ybd^n, \bar\etabd^n, \varphibd^n, \widetilde\Bbd^n, \widetilde\etabd^n) \to (\bar\Xbd, \bar\Ybd, \bar\etabd, \varphibd, \zero, \zero)$ in $\Dmb([0,T]:\bell_1^\downarrow \times \bell_1 \times \Rmb^\infty) \times S_{M_0} \times (\Dmb([0,T]:\Rmb))^\infty \times (\Dmb([0,T]:\Rmb))^\infty$ a.s.\ on $(\Omega^*,\Fmc^*,\Pmb^*)$, and thus the rest of the argument will be made a.s.\ on $(\Omega^*,\Fmc^*,\Pmb^*)$.
	From the $\Cmb$-tightness proved in Lemma \ref{lem:tightness}, $(\bar\Xbd,\bar\Ybd, \bar\etabd)$ takes values in $\Cmb([0,T]:\bell_1^\downarrow \times \bell_1 \times \Rmb^\infty)$.
	
	We first prove part (a).
	Using the triangle inequality, for each $i \ge 1$,
	\begin{align}
		&\left|\int_{\Xt}\one_{[0,\bar{X}^n_i(s)-\bar{X}^n_{i+1}(s))}(y)\varphi^n_i(s,y)\,ds\,dy-\int_{\Xt}\one_{[0,\bar{X}_i(s)-\bar{X}_{i+1}(s))}(y)\varphi_i(s,y)\,ds\,dy\right| \notag \\
		& \leq \int_{\Xt}|\one_{[0,\bar{X}^n_i(s)-\bar{X}^n_{i+1}(s))}(y)-\one_{[0,\bar{X}_i(s)-\bar{X}_{i+1}(s))}(y)|\varphi^n_i(s,y)\,ds\,dy \notag \\
		&\qquad + \left|\int_{\Xt}\one_{[0,\bar{X}_i(s)-\bar{X}_{i+1}(s))}(y) (\varphi^n_i(s,y)-\varphi_i(s,y))\,ds\,dy\right|.
		\label{eq:cvg-1}
	\end{align}
	Since $\leb_t\{(s,y) : y=\bar{X}_i(s)-\bar{X}_{i+1}(s) \}=0$, where $\leb_t$ is the Lebesgue measure on $[0,t]\times[0,1]$, we have
	\begin{equation*}
		|\one_{[0,\bar{X}^n_i(s)-\bar{X}^n_{i+1}(s))}(y)-\one_{[0,\bar{X}_i(s)-\bar{X}_{i+1}(s))}(y)|\to 0
	\end{equation*}
	as $n \to \infty$ for $\leb_t$-a.e.\ $(s,y)\in[0,t]\times[0,1]$.
	From \eqref{eqn:controlBound} and the super-linearity of $\ell$, one has the uniform integrability of
		$(s,y)\mapsto\varphi^n_i(s,y)$
	with respect to the normalized Lebesgue measure on $[0,T]\times[0,1]$.
	The above two observations imply that, as $n\to \infty$,
	\begin{equation}
		\label{eq:cvg-2}
		\int_{\Xt}|\one_{[0,\bar{X}^n_i(s)-\bar{X}^n_{i+1}(s))}(y)-\one_{[0,\bar{X}_i(s)-\bar{X}_{i+1}(s))}(y)|\varphi^n_i(s,y)\,ds\,dy\to 0.
	\end{equation}
	Recalling the topology on $S_{M_0}$, the convergence $\varphibd^n \to \varphibd$ and $\lambda_n\to \lambda$ implies that
	\begin{align}
		& \left|\int_{\Xt}\one_{[0,\bar{X}_i(s)-\bar{X}_{i+1}(s))}(y) (\varphi^n_i(s,y)-\varphi_i(s,y))\,ds\,dy\right|
		\to 0, \label{eq:cvg-3} \\
		& \lambda_n\int_{\Xt}\varphi^n_{0}(s,y)\,ds\,dy \to \lambda\int_{\Xt}\varphi_{0}(s,y)\,ds\,dy \label{eq:cvg-4}
	\end{align}
    as $n \to \infty$.
	Combining \eqref{eq:Ybar-decom-1}, \eqref{eq:Ybar-decom-i} with \eqref{eq:cvg-0}--\eqref{eq:cvg-4} completes the proof of part (a).

	We now prove part (b).
	The fact that $\varphibd \in \Smc(\bar\Xbd, \bar\Ybd)$ will be immediate from part (a) once we have $(\bar\Xbd, \bar\Ybd)\in\Cmc$.
	Since $(\bar\Xbd, \bar\Ybd)\in\Cmb([0,T]:\bell_1^\downarrow \times \bell_1)$, in order to show $(\bar\Xbd, \bar\Ybd)\in\Cmc$, it suffices to verify properties (i) and (ii) in the definition of $\Cmc$.
	
	\textit{Verification of property (ii)}:
	The validity of \eqref{eq:XbarLim1}--\eqref{eq:XbarLim2} is immediate from the fact that these equalities hold with
	$(\bar\Xbd, \bar\Ybd, \bar\etabd)$ replaced with
	$(\bar\Xbd^n, \bar\Ybd^n, \bar\etabd^n)$.
	Fix $i\in \Nmb$.
	Clearly $\etabar_i(0)=0$ and $\etabar_i(\cdot)$ is nondecreasing since $\etabar_i^n(0)=0$ and $\etabar_i^n(\cdot)$ is nondecreasing.
	It remains to verify that
	\begin{equation}
		\label{eqn:reflectionproperty}
		\int_0^T\left(1-\bar{X}_i(s)\right)\etabar_i(ds)=0.
	\end{equation}
	Recall the compensated process $\etatil^n_i(t)$ defined in \eqref{eq:etatil} and estimated in \eqref{eq:cvg-0}. 
	Then
	\begin{equation*}
		\lambda_n \int_{\Xt} \one_{[0,\beta_n(\bar{X}^n_i(s)))}(y) \varphi^n_0(s,y)\,ds\,dy = \etabar^n_i(t) - \etatil^n_i(t) \to \etabar_i(t)
	\end{equation*}
	uniformly in $t \in [0,T]$, by $\Cmb$-tightness of $\etabar^n_i$ and the convergence that $(\etabar^n_i,\etatil^n_i) \to (\etabar_i,0)$.
	Using this and the fact that $s \mapsto \Xbar_i(s)$ is bounded and continuous, we have
	\begin{align*}
		\int_0^T\left(1-\bar{X}_i(s)\right)\etabar_i(ds) & = \lim_{n \to \infty} \int_0^T\left(1-\bar{X}_i(s)\right)[\etabar^n_i - \etatil^n_i](ds) \\
		& = \lim_{n \to \infty} \int_{\XT} \left(1-\bar{X}_i(s)\right) \lambda_n \one_{[0,\beta_n(\bar{X}^n_i(s)))}(y) \varphi^n_0(s,y)\,ds\,dy.
	\end{align*}
	For any $K>0$,
	\begin{align*}
		& \limsup_{n \to \infty} \int_{\XT} \left(1-\bar{X}_i(s)\right) \lambda_n \one_{[0,\beta_n(\bar{X}^n_i(s)))}(y) \varphi^n_0(s,y)\,ds\,dy \\
		& \le \limsup_{n \to \infty} \lambda_n \int_{\XT} \left(1-\bar{X}_i(s)\right) \one_{[0,\beta_n(\bar{X}^n_i(s)))}(y) [K+\gamma(K)\ell(\varphi^n_0(s,y))]\,ds\,dy \\
		& \le \limsup_{n \to \infty} \lambda_n K \int_{\XT} \left(1-\bar{X}_i(s)\right) \one_{[0,\beta_n(\bar{X}^n_i(s)))}(y) \,ds\,dy + \limsup_{n \to \infty} \lambda_n \gamma(K) M_0 \\
		& = \lambda K \int_{\XT} \left(1-\bar{X}_i(s)\right) \lim_{n \to \infty}  \one_{[0,\beta_n(\bar{X}^n_i(s)))}(y) \,ds\,dy + \lambda \gamma(K) M_0,
	\end{align*}
	where the second line uses Lemma \ref{lem:ellProp}(a), the third line uses \eqref{eqn:controlBound}, and the last line uses the dominated convergence theorem.
	Since $\beta_n(x) \le x^{d_n}$ for $x \in [0,1]$, we have $\beta_n(x_n) \to 0$ whenever $\limsup_{n \to \infty} x_n <1$. 
	Since $\Xbar_i^n(s) \to \Xbar_i(s)$ for each $s \in [0,T]$, we must have $$\int_{\XT} \left(1-\bar{X}_i(s)\right) \lim_{n \to \infty}  \one_{[0,\beta_n(\bar{X}^n_i(s)))}(y) \,ds\,dy = 0.$$
	Since $\gamma(K) \to 0$ as $K \to \infty$, we have verified property (ii) in the definition of $\Cmc$.
	
	\textit{Verification of property (i)}:
	From part (a) it is clear that $\Ybar_i(0)=x_i$ and $\Ybar_i$ is absolutely continuous on $[0,T]$ for each $i$.
	From property (ii) and properties of the Skorokhod map $\Gamma_\infty$ in Remark \ref{rmk:Skorokhod-map}, we have that $\Xbar_i(0)=x_i$ and $\Xbar_i$ is absolutely continuous on $[0,T]$ for each $i$.
	This verifies property (i) and completes the proof. 
\end{proof}

\section{Laplace Upper Bound}\label{sec:UpperBound}

This section is devoted to the proof of the Laplace upper bound \eqref{eqn:LaplaceUpperBound}.
Fix $G \in \Cmb_b(\pathspace)$.
From the variational representation in Lemma \ref{lem:varRep}, for all $n\in\Nmb$, we can select a control $\ti\varphibd^n\in\bar{\controlset}_b$ such that
\begin{equation}\label{eqn:upperNopt}
	-\frac{1}{n}\log \Emb e^{-nG(\Xbd^n,\Ybd^n)}\geq \Emb\left\{\sum_{i=0}^\infty\int_{\XT}\vartheta^n_i\ell(\ti\varphi^n_i(s,y))\,ds\,dy + G(\bar{\Xbd}^{n,\ti\varphibd^n},\bar{\Ybd}^{n,\ti\varphibd^n})\right\}-\frac{1}{n}.
\end{equation}
This shows that
\begin{equation*}
	\sup_{n\in\Nmb}\Emb\sum_{i=0}^\infty\int_{\XT}\vartheta^n_i\ell(\ti\varphi^n_i(s,y))\,ds\,dy
	\leq 2\|G\|_\infty+1
	=: M_G.
\end{equation*}
By a standard localization argument (see e.g.\ \cite[Proof of Theorem 4.2]{budhiraja2011variational}) and since $\lambda^n \to \lambda >0$, it now follows that for any fixed $\sigma>0$ there is an $M_0 \in (0,\infty)$
and a sequence $\varphibd^n \in \bar \controlset_b$ taking values in $S_{M_0}$ a.s.\ such that, for all $n$, the expected value on the right side of \eqref{eqn:upperNopt} differs from the same expected value, but with $\tilde \varphibd^n$ replaced by $\varphibd^n$ throughout, by at most $\sigma$. In particular,
\begin{equation}\label{eqn:UpperBound1}
	-\frac{1}{n}\log \Emb e^{-nG(\Xbd^n,\Ybd^n)}\geq \Emb\left\{\sum_{i=0}^\infty\int_{\XT}\vartheta^n_i\ell(\varphi^{n}_i(s,y))\,ds\,dy + G(\bar{\Xbd}^{n,\varphibd^{n}},\bar{\Ybd}^{n,\varphibd^{n}})\right\}-\frac{1}{n}-\sigma.
\end{equation}
Now we can complete the proof of the Laplace upper bound.
Since $\varphibd^n$ are in $S_{M_0}$ a.s., from Lemma  \ref{lem:tightness} we have the tightness of $\{(\bar\Xbd^n,\bar\Ybd^n,\bar\etabd^n, \varphibd^n)\}_{n \in \Nmb}$.
Assume without loss of generality that $\{(\bar\Xbd^n,\bar\Ybd^n,\bar\etabd^n, \varphibd^n)\}_{n \in \Nmb}$ converges along the whole sequence, in distribution, to $(\bar\Xbd,\bar\Ybd,\bar\etabd, \varphibd)$, given on some probability space $(\Omega^*, \Fmc^*, \Pmb^*)$.
By Lemma \ref{lem:convergence} we have $(\bar\Xbd, \bar\Ybd) \in \Cmc$ and  $\varphibd \in \Smc(\bar\Xbd, \bar\Ybd)$ a.s.\ $\Pmb^*$.
Using \eqref{eqn:UpperBound1}, Fatou's lemma, and the definition of $\rate$ in \eqref{eqn:JSQRateFunction}
\begin{align*}
	& \liminf_{n\to\infty} -\frac{1}{n}\log \Emb e^{-nG(\Xbd^n,\Ybd^n)} \\
	& \ge \liminf_{n\to\infty}\Emb\left\{\sum_{i=0}^\infty\int_{\XT}\vartheta_i^n\ell(\varphi^n_i(s,y))\,ds\,dy + G(\bar{\Xbd}^n,\bar{\Ybd}^n)-\frac{1}{n}-\sigma\right\} \\
	& \ge \Emb^*\left\{\sum_{i=0}^\infty\int_{\XT}\vartheta_i\ell(\varphi_i(s,y))\,ds\,dy + G(\bar{\Xbd},\bar{\Ybd})\right\}-\sigma\\
	& \ge \inf_{(\zetabd,\psibd)\in\Cmc} \{\rate(\zetabd,\psibd) + G(\zetabd,\psibd)\}-\sigma,
\end{align*}
where the second inequality is a consequence of  a lower semicontinuity property of $\ell$,  cf. \cite[Lemma A.1]{budhiraja2013large}.
Since $\sigma \in (0,1)$ is arbitrary, this completes the proof of the Laplace upper bound.
\qed

\section{Laplace Lower Bound}\label{sec:LowerBound}

This section is devoted to the proof of the Laplace lower bound \eqref{eqn:LaplaceLowerBound}. 
The following lemma, adapted from \cite[Lemma 5.1]{BudhirajaFriedlanderWu2021many}, is key to the proof of the lower bound \eqref{eqn:LaplaceLowerBound}. It says that,  given a trajectory $(\zetabd^*,\psibd^*) \in \Cmc$, one can select a trajectory $(\zetabd,\psibd)$ which is suitably close to $(\zetabd^*,\psibd^*)$ and a control $\varphibd$ such that $(\zetabd,\psibd)$ is the unique trajectory driven by $\varphibd$.
We note that although \cite[Lemma 5.1(a)]{BudhirajaFriedlanderWu2021many} is stated with respect to the product topology on $\Cmb([0,T]:\Rmb^\infty \times \Rmb^\infty)$, the result actually holds for $\Cmb([0,T]:\bell_1 \times \bell_1)$ with the corresponding norm denoted by 
$$\|(\zetabd,\psibd)\|_{1,\infty} := \sup_{0 \le t \le T} \left( \|\zetabd(t)\|_1 + \|\psibd(t)\|_1 \right), \quad (\zetabd,\psibd) \in \Cmb([0,T]:\bell_1 \times \bell_1),$$ 
as stated in Lemma \ref{lem:uniqueness}(a) below.
More details on this are provided in Appendix \ref{sec:appendix}.

\begin{Lemma}\label{lem:uniqueness}
	Fix $\sigma\in(0,1)$. Given $(\zetabd^*,\psibd^*)\in\Cmc$ with $\rate(\zetabd^*,\psibd^*)<\infty$, there exists $(\zetabd,\psibd)\in\Cmc$ and $\varphibd\in\Smc(\zetabd,\psibd)$ such that
	\begin{enumerate}[(a)]
	\item
		$\|(\zetabd,\psibd)-(\zetabd^*,\psibd^*)\|_{1,\infty} \le \sigma$.
	\item
		$\sum_{i=0}^\infty\int_{\XT} \vartheta_i\ell(\varphi_i(s,y)) \,ds\,dy \le \rate(\zetabd,\psibd) +\sigma \leq \rate(\zetabd^*,\psibd^*)+2\sigma$.
	\item
		If $(\tilde\zetabd,\tilde\psibd)$ is another pair in $\Cmc$ such that $\varphibd\in\Smc(\tilde\zetabd,\tilde\psibd)$, then $(\tilde\zetabd,\tilde\psibd)=(\zetabd,\psibd)$.
	\end{enumerate}
\end{Lemma}

We now complete the proof of the lower bound using this result.
Fix $G\in\Cmb_b(\pathspace)$ and $\sigma\in(0,1)$. Select a trajectory $(\zetabd^*,\psibd^*)$ which is $\sigma$-optimal for the RHS of \eqref{eqn:LaplaceLowerBound}, namely
\begin{equation}
	\label{eq:lowerbd_pf1}
	\rate(\zetabd^*,\psibd^*)+G(\zetabd^*,\psibd^*) \leq \inf_{(\zetabd,\psibd)\in\Cmc}\{\rate(\zetabd,\psibd)+G(\zetabd,\psibd)\}+\sigma.
\end{equation}
By continuity of $G$ and Lemma \ref{lem:uniqueness}, we can find $(\bar\zetabd,\bar\psibd) \in \Cmc$ and $\bar\varphibd \in S_T(\bar\zetabd,\bar\psibd)$ such that the uniqueness property in  Lemma \ref{lem:uniqueness} holds (with $\varphibd$ replaced by $\bar \varphibd$) and 
\begin{align}
	\sum_{i=0}^\infty\int_{\XT}\vartheta_i\ell(\varphibar_i(s,y))\,ds\,dy+G(\bar\zetabd,\bar\psibd) & \le \rate(\bar\zetabd,\bar\psibd)+G(\bar\zetabd,\bar\psibd)+\sigma \notag \\
	& \leq \rate(\zetabd^*,\psibd^*)+G(\zetabd^*,\psibd^*)+2\sigma. \label{eq:lowerbd_pf2}
\end{align}
Consider the controlled system \eqref{eq:controlled-process} with
control $\varphibd^n \in \bar{\controlset}_b$ given by
\begin{align*}
	\varphi^n_i(s,y) & := \frac{1}{n} \one_{\{\varphibar_i(s,y) \le \frac{1}{n}\}} + \varphibar_i(s,y) \one_{\{\frac{1}{n} < \varphibar_i(s,y) < n\}} + n \one_{\{\varphibar_i(s,y) \ge n \}}, \quad i \le n, \\
	\varphi^n_i(s,y) & := 1, \quad i > n.
\end{align*}
Then there is an $M_0 \in (0,\infty)$ such that the sequence $\{\varphibd^n\}$ satisfies \eqref{eqn:controlBound}.
Furthermore, it is easily checked that
 $\varphibd^n \to \bar\varphibd$ (in $S_{M_0}$).
It then follows from Lemmas \ref{lem:tightness} and \ref{lem:convergence} that $\{(\bar\Xbd^n, \bar\Ybd^n, \bar\etabd^n, \varphibd^n)\}_{n \in \Nmb}$ is tight and any limit point $(\bar\Xbd,\bar\Ybd,\bar\etabd,\varphibd)$, given on some probability space $(\Omega^*,\Fmc^*,\Pmb^*)$,
satisfies $(\bar\Xbd, \bar\Ybd) \in \Cmc$ and $\varphibd \in \Smc(\bar\Xbd, \bar\Ybd)$ a.s.\ $\Pmb^*$.
From the fact that $\varphibd^n\to\bar\varphibd$ we must have $\varphibd=\bar\varphibd$.
Thus $\bar\varphibd\in\Smc(\bar\Xbd, \bar\Ybd)$  and since we also have $\bar\varphibd\in\Smc(\bar\zetabd,\bar\psibd)$, we must have  $(\bar\Xbd, \bar\Ybd)=(\bar\zetabd,\bar\psibd)$ a.s.\ $\Pmb^*$ from the uniqueness property noted above.
Noting that $\ell(\varphi^n_i(s,y))\leq \ell(\varphibar_i(s,y))$ for all $n\in\Nmb$ and $(s,y)\in[0,T]\times[0,1]$,
it then follows from the variational representation \eqref{eqn:variationalRep} and \eqref{eq:lowerbd_pf1}--\eqref{eq:lowerbd_pf2} that
\begin{align*}
	\limsup_{n\to\infty}-\frac{1}{n}\log\Emb e^{-nG(\Xbd^n,\Ybd^n)} 
	&\leq \limsup_{n\to\infty}\Emb\left\{\sum_{i=0}^\infty\int_{\XT}\vartheta^n_i\ell(\varphi^n_i(s,y))\,ds\,dy+G(\bar{\Xbd}^{n},\bar{\Ybd}^{n})\right\}\\
	&\leq \sum_{i=0}^\infty\int_{\XT}\vartheta_i\ell(\varphibar_i(s,y))\,ds\,dy+G(\bar\zetabd,\bar\psibd)\\
	&\leq\inf_{(\zetabd,\psibd)\in\Cmc}\left\{\rate(\zetabd,\psibd) + G(\zetabd,\psibd)\right\}+3\sigma.
\end{align*}
The inequality in \eqref{eqn:LaplaceLowerBound} now follows upon sending $\sigma\to0$. \qed

\section{Compact Sub-level Sets}
\label{sec:compactSets}

In this section we prove the third statement in the proof of Theorem \ref{thm:mainResult}, namely the property that $\rate$ is a rate function.
For this we need to show that for every $M \in \Nmb$, the set $\Upsilon_M := \{(\zetabd,\psibd)\in \pathspace:
\rate(\zetabd,\psibd)\le M\}$ is compact.
Now fix such an $M$ and a sequence $\{(\zetabd^n,\psibd^n)\} \subset \Upsilon_M$. It suffices to show that the sequence has a convergent subsequence with the limit in the set $\Upsilon_M$.
From the definition of $\rate$, it follows that $(\zetabd^n,\psibd^n)\in\Cmc$
and there exists a control $\varphibd^n\in\Smc(\zetabd^n,\psibd^n)$ such that for every $n \in \Nmb$
\begin{equation}
	\label{eqn:controlBoundCompact}
	\sum_{i=0}^\infty\int_{\XT} \vartheta_i\ell(\varphi^n_i(s,u))\,ds\,dy
	\leq \rate(\zetabd^n,\psibd^n)+\frac{1}{n}
	\leq M+\frac{1}{n}.
\end{equation}
We follow the convention that $\zeta^n_0 = 1$.
Recall the compact metric spaces $S_N$, for $N \in \Nmb$, introduced in Section \ref{sec:tight}.
From \eqref{eqn:controlBoundCompact} we have $\varphibd^n \in S_{M_0}$ with $M_0:=(1+\lambda^{-1})(M+1)$.
We first show pre-compactness of the sequence $\{(\zetabd^n,\psibd^n,\varphibd^n)\}_{n\in\Nmb}$. 

\begin{Lemma}
	\label{lem:tightCompact}
	The sequence
	$\{(\zetabd^n,\psibd^n,\varphibd^n)\}_{n\in\Nmb}$ is pre-compact in $\Cmb([0,T]:\bell_1^\downarrow \times \bell_1) \times S_{M_0}$.
\end{Lemma}

\begin{proof}
	Pre-compactness of $\{\varphibd^n\}_{n\in\Nmb_0}$ is  immediate from the compactness of $S_{M_0}$.
	
	We next prove pre-compactness of $\{\psibd^n(t)\}_{n\in\Nmb}$ in $\bell_1$ for fixed $t \in [0,T]$.
	It suffices to show that $\{\psi^n_i(t)\}_{n \in \Nmb}$ is pre-compact for each $i \in \Nmb$ and that
    \begin{equation}
		\label{eq:compact-criterion}
		\lim_{k \to \infty} \sup_{n \in \Nmb} \sum_{i=k}^\infty |\psi^n_i(t)| = 0.
	\end{equation}
    Now fix $i \in \Nmb$.
    From \eqref{eq:psi1} and \eqref{eq:psii} we have
    \begin{align*}
        |\psi^n_i(t)| & \le x_i + \int_{\Xt} [\lambda\varphi_0^n(s,y) + \varphi_i^n(s,y)]\,ds\,dy \\
        & \le x_i + \int_{\Xt} [\lambda(\ell(\varphi_0^n(s,y))+2) + \ell(\varphi_i^n(s,y))+2]\,ds\,dy \\
        & \le 1+(M+1)+2(\lambda+1)T,
    \end{align*}
    where the second inequality uses Lemma \ref{lem:ellProp}(b) and the last inequality uses \eqref{eqn:controlBoundCompact}.
    So we have pre-compactness of $\{\psi^n_i(t)\}_{n \in \Nmb}$.
	To show \eqref{eq:compact-criterion}, first note that by \eqref{eq:zeta_psi_eta} and non-negativity of $\zetabd^n(t)$, we have
	\begin{equation}
		\label{eq:zeta-bd}
		\|\zetabd^n(t)\|_1 = \sum_{i=1}^\infty \psi^n_i(t) \le \|\xbd\|_1 + \lambda \int_{\Xt} \varphi_0(s,y)\,dsdy \le \|\xbd\|_1 + (M+1) + 2\lambda T,
	\end{equation}
	where the first inequality uses \eqref{eq:psi1}--\eqref{eq:psii} and the last inequality uses Lemma \ref{lem:ellProp}(b) and \eqref{eqn:controlBoundCompact}.  
	Again from \eqref{eq:psii} we have, for any $K>0$, 
	\begin{align*}
		\limsup_{k \to \infty} \sup_{n \in \Nmb} \sum_{i=k}^\infty |\psi^n_i(t)| 
		& \le \limsup_{k \to \infty} \sup_{n \in \Nmb} \sum_{i=k}^\infty \left(x_i + \int_{\Xt} \one_{[0,\zeta_i^n(s)-\zeta_{i+1}^n(s))}(y) \varphi_i^n(s,y)\,ds\,dy \right) \\
		& \le \limsup_{k \to \infty} \sup_{n \in \Nmb} \sum_{i=k}^\infty \int_{\Xt} \one_{[0,\zeta_i^n(s)-\zeta_{i+1}^n(s))}(y) [K + \gamma(K)\ell(\varphi_i^n(s,y))]\,ds\,dy \\
		& \le \limsup_{k \to \infty} \sup_{n \in \Nmb} K \int_0^t \zeta_k^n(s)\,ds + \lim_{k \to \infty} \sup_{n \in \Nmb} \gamma(K) (M+\frac{1}{n}) \\
		& \le \lim_{k \to \infty} \sup_{n \in \Nmb} K \int_0^t \frac{\|\zeta^n(s)\|_1}{k}\,ds + \gamma(K) (M+1) \\
		& = \gamma(K) (M+1) 
	\end{align*}
	which converges to $0$ as $K \to \infty$.
	Here the second line follows from $\xbd \in \bell_1$ and Lemma \ref{lem:ellProp}(a), the third uses \eqref{eqn:controlBoundCompact}, the fourth uses the monotonicity of $k\mapsto \zeta^n_k(t)$, and the last uses \eqref{eq:zeta-bd}.
	This gives \eqref{eq:compact-criterion} and the pre-compactness of $\{\psibd^n(t)\}_{n\in\Nmb}$ in $\bell_1$.
	
	Next we show that  $\{\psibd^n\}$ is equicontinuous. 
	Note that for any $0 < t-s \le  \delta$ and $K>0$,
	\begin{align*}
		\|\psibd^n(t)-\psibd^n(s)\|_1 
		&\leq \lambda \int_{[s,t]\times[0,1]}\varphi^n_0(u,y)\,du\,dy + \sum_{i=1}^\infty \int_{[s,t]\times[0,1]} \one_{[0,\zeta_i^n(t)-\zeta_{i+1}^n(t))}(y) \varphi^n_i(u,y) \,du\,dy\\
		&\leq \lambda \int_{[s,t]\times[0,1]} [K + \gamma(K)\ell(\varphi^n_0(u,y))]\,du\,dy \\
		&\qquad +\sum_{i=1}^\infty \int_{[s,t]\times[0,1]} \one_{[0,\zeta_i^n(t)-\zeta_{i+1}^n(t))}(y) [K + \gamma(K)\ell(\varphi^n_i(u,y))] \,du\,dy \\
		&\leq (\lambda+1)K\delta + \gamma(K)(M+1),
	\end{align*}
	where the second line uses Lemma \ref{lem:ellProp}(a) and the last uses \eqref{eqn:controlBoundCompact}.
	Therefore,
	\begin{equation*}
		\limsup_{\delta\to0}\sup_{n\in\Nmb}\sup_{|t-s|\leq \delta} \|\psibd^n(t)-\psibd^n(s)\|_1 \leq \gamma(K)(M+1)
	\end{equation*}
	and the equicontinuity of $\{\psibd^n\}$ follows upon sending $K\to\infty$. 
	
	Using the Arzela-Ascoli Theorem, we have pre-compactness of $\{\psibd^n\}_{n\in\Nmb}$ in $\Cmb([0,T]:\bell_1)$. 
	Let $L > \|\xbd\|_1 + (M+1) + 2\lambda T$ be an integer.
	From \eqref{eq:zeta-bd} we see that 
	$$\sup_{n \in \Nmb} \sup_{0 \le t \le T} \zeta^n_L(t) \le \sup_{n \in \Nmb} \sup_{0 \le t \le T} \|\zetabd^n(t)\|_1 / L < 1.$$
	Therefore for each $n \in \Nmb$, $\eta_i^n \equiv 0$ for $i \ge L$, $\zeta_i^n = \psi_i^n$ for $i > L$, and $(\zeta^n_i,\eta^n_i)_{1 \le i \le L}$ is the unique solution to the finite-dimensional Skorokhod problem for $(\psi^n_i)_{1 \le i \le L}$ associated with the reflection matrix $R_L$. 
	In particular, 
	$$\zeta_i^n(t) = \psi_i^n(t) + \eta_{i-1}^n(t) - \eta_i^n(t), \quad i < L; \qquad \zeta_L^n(t) = \psi_L^n(t) + \eta_{L-1}^n(t).$$
	So pre-compactness of $\{(\zetabd^n,\psibd^n)\}_{n\in\Nmb}$ in $\Cmb([0,T]:\bell_1^\downarrow \times \bell_1)$ follows immediately from the pre-compactness of $\{\psibd^n\}_{n\in\Nmb}$ and the Lipschitz property in \eqref{eq:Skorokhod-Lipschitz}.
\end{proof}

We now return to the proof of compactness of $\Upsilon_M$. 
Consider a sequence $\{(\zetabd^n,\psibd^n)\}_{n \in \Nmb} \subset \Upsilon_M$. Then Lemma \ref{lem:tightCompact} shows that such a sequence is pre-compact. 
It then follows from \cite[Lemma 6.2]{BudhirajaFriedlanderWu2021many} that any limit point
$(\zetabd,\psibd)$ of $\{(\zetabd^n,\psibd^n)\}_{n \in \Nmb}$ is in $\Upsilon_M$. 
This establishes the desired compactness.  \qed

\section{Bounds on Probabilities of Long Queues}
\label{sec:examples}
In this section we prove Theorem \ref{thm:large-customers}. 
Fix $\varepsilon>0$ and recall the notation $G_\varepsilon^n, F_\varepsilon^n$ from the statement of the theorem. 
Since $\Imc(\zetabd,\psibd) = \infty$ for $(\zetabd,\psibd) \notin \Cmc$, we define the following (relatively) open and closed sets in $\Cmc$ for $\delta > 0$: 
\begin{equation*}
	G_\delta := \{(\zetabd,\psibd) \in \Cmc : \|\zetabd\|_{1,\infty} > \|\xbd\|_1 + \delta\}, \quad
	F_\delta := \{(\zetabd,\psibd) \in \Cmc : \|\zetabd\|_{1,\infty} \ge \|\xbd\|_1 + \delta\}.
\end{equation*}
In order to prove the first statement in the theorem we first evaluate $\Imc(F_\varepsilon)$, where $\Imc(A):=\inf_{(\zetabd,\psibd) \in A} \Imc(\zetabd,\psibd)$ for $A \subset \pathspace$. 

As a preparation for evaluating $\Imc(F_\varepsilon)$, we state and prove the following well-posedness result for trajectories driven by a bounded arrival control $\alpha$ and a bounded and almost continuous ``master'' service control $\theta$.

\begin{Lemma}
	\label{lem:special-well-posedness}
	Suppose $\xbd \in \bell_1^\downarrow$ and $\alpha, \theta \colon [0,T] \times [0,1] \to \Rmb_+$ satisfies
	$$\int_{\XT} \ell(\alpha(s,y))\,ds\,dy < \infty, \quad \|\theta\|_\infty:= \sup_{(s,y) \in [0,T] \times [0,1]} \theta(s,y) < \infty,$$
    and that $\theta(s,y)$ is continuous at a.e.\ $y \in [0,1]$ for each $s \in [0,T]$.
	Then there exists a unique pair $(\zetabd,\psibd) \in \Cmc$ such that
	\begin{align}
		\psi_1(t) & = x_1 + \lambda \int_{\Xt}\alpha(s,y)\,ds\,dy  - \int_{\Xt}\one_{[1-\zeta_1(s),1-\zeta_2(s))}(y) \theta(s,y) \,ds\,dy, \label{eq:special-psi1}\\
		\psi_i(t) & = x_i - \int_{\Xt}\one_{[1-\zeta_i(s),1-\zeta_{i+1}(s))}(y) \theta(s,y) \,ds\,dy, \quad i \ge 2. \label{eq:special-psii}
	\end{align}		
	In particular, $\varphibd \in \Smc(\zetabd,\psibd)$ where $\varphi_0=\alpha$ and 
	\begin{equation*}
		\varphi_i(s,y) = \theta(s,y+1-\zeta_i(s)) \one_{[0,\zeta_i(s)-\zeta_{i+1}(s))}(y) + \one_{[\zeta_i(s)-\zeta_{i+1}(s),1)}(y), \quad i \ge 1.
	\end{equation*}
\end{Lemma}

\begin{proof}
	We first show uniqueness.
	Suppose there are two such pairs $(\zetabd,\psibd), (\bar\zetabd,\bar\psibd) \in \Cmc$. Denote the corresponding reflection terms by $\etabd$ and $\bar \etabd$, respectively.
	Note that $C:=\|\xbd\|_1+\lambda\int_{\XT} \alpha(s,y)\,ds\,dy < \infty$ by Lemma \ref{lem:ellProp}(b).
	Let $K:=\lceil C +1 \rceil \in \Nmb$. 
	Then there is no reflection for coordinates $i \ge K$, i.e. $\eta_i= \bar \eta_i = 0$ for all $i\ge K$, and hence
	$$\sum_{i=K+1}^\infty |\zeta_i(t)-\zetabar_i(t)| = \sum_{i=K+1}^\infty |\psi_i(t)-\psibar_i(t)|.$$
	For coordinates $i \le K$, using the $C_K$-Lipschitz property in \eqref{eq:Skorokhod-Lipschitz}, we have
	$$\sum_{i=1}^{K} |\zeta_i(t)-\zetabar_i(t)| \le C_K \sum_{i=1}^{K} |\psi_i(t)-\psibar_i(t)|.$$
	Therefore
	\begin{align*}
		\sum_{i=1}^\infty |\zeta_i(t)-\zetabar_i(t)| & \le (1+C_K) \sum_{i=1}^\infty |\psi_i(t)-\psibar_i(t)| \\
		& \le (1+C_K) \|\theta\|_\infty\sum_{i=1}^\infty \int_0^t  (|\zeta_i(s)-\zetabar_i(s)|+|\zeta_{i+1}(s)-\zetabar_{i+1}(s)|)\,ds \\
		& \le 2(1+C_K)\|\theta\|_\infty \int_0^t \sum_{i=1}^\infty |\zeta_i(s)-\zetabar_i(s)|\,ds.
	\end{align*}
	Using Gronwall's lemma we get $\sum_{i=1}^\infty |\zeta_i(t)-\zetabar_i(t)| = 0$ for $t \in [0,T]$, namely $\zetabd = \bar\zetabd$.
    From \eqref{eq:special-psi1}-\eqref{eq:special-psii} it then follows that $\psibd = \bar \psibd$ as well.
	This gives uniqueness.
	
	Now we show existence.
	Consider the controlled system \eqref{eq:controlled-process} with controls $\varphibd^n \in \Amcbar_b$ given by 
	\begin{align*}
		\varphi_0^n(s,y) & = \frac{1}{n} \one_{\{\alpha(s,y) \le \frac{1}{n}\}} + \alpha(s,y) \one_{\{\frac{1}{n} < \alpha(s,y) < n\}} + n \one_{\{\alpha(s,y) \ge n \}}, \\ 
        \varphi_i^n(s,y) & = 1, \quad i > n, \\
		\varphi_i^n(s,y) & = \max\{\frac{1}{n}, \theta(s,y+1-\Xbar^{n}_i(s-))\} \one_{[0,\Xbar^{n}_i(s-)-\Xbar^{n}_{i+1}(s-))}(y) \\
		& \quad + \one_{[\Xbar^{n}_i(s-)-\Xbar^{n}_{i+1}(s-),1)}(y), \quad 1 \le i \le n.
	\end{align*}
	Note that $\varphi_i^n$'s make use of values from the ``master'' control $\theta$ within the disjoint $y$-intervals $[1-\Xbar^{n}_i(s-),1-\Xbar^{n}_{i+1}(s-))$.
	Then there is an $M_0 \in (0,\infty)$ such that the sequence $\{\varphibd^n\}$ satisfies \eqref{eqn:controlBound}.
	It then follows from Lemmas \ref{lem:tightness} and \ref{lem:convergence} that the sequence $\{(\bar\Xbd^n, \bar\Ybd^n, \bar\etabd^n, \varphibd^n)\}_{n \in \Nmb}$ is tight and any limit point $(\bar\Xbd,\bar\Ybd,\bar\etabd,\varphibd)$, given on some probability space $(\Omega^*,\Fmc^*,\Pmb^*)$,
	satisfies $(\bar\Xbd, \bar\Ybd) \in \Cmc$ and $\varphibd \in \Smc(\bar\Xbd, \bar\Ybd)$ a.s.\ $\Pmb^*$.
	From the construction of $\varphibd^n$ and the continuity of $\theta(s,y)$ in  a.e.\ $y$, we must have $\varphi_0=\alpha$ and 
	\begin{equation*}
		\varphi_i(s,y) = \theta(s,y+1-\Xbar_i(s)) \one_{[0,\Xbar_i(s)-\Xbar_{i+1}(s))}(y) + \one_{[\Xbar_i(s)-\Xbar_{i+1}(s),1)}(y), \quad i \ge 1, \mbox{ a.e.\ } (s,y) \in \mathbb{X}_T.
	\end{equation*}
	Noting that by a shifting in $y$, we have
	\begin{equation*}
		\int_{\Xt}\one_{[0,\Xbar_i(s)-\Xbar_{i+1}(s))}(y) \varphi_{i}(s,y) \,ds\,dy = \int_{\Xt}\one_{[1-\Xbar_i(s),1-\Xbar_{i+1}(s))}(y) \theta(s,y) \,ds\,dy. 
	\end{equation*}
	Therefore \eqref{eq:special-psi1} and \eqref{eq:special-psii} are satisfied with $(\zetabd,\psibd)$.
	This gives existence and completes the proof.
\end{proof}

Now we are ready to obtain the precise expression of $\Imc(F_\varepsilon)$.
We will first give a candidate optimal trajectory $(\zetabd^*,\psibd^*) \in F_\varepsilon$ and then show that it is indeed optimal.
Let
\begin{equation}
	\label{eq:optimizer}
	a^* :=  \frac{\frac{\varepsilon}{T}+\sqrt{4+(\frac{\varepsilon}{T})^2}}{2} > 1, \quad b^* := 1/a^* = \frac{-\frac{\varepsilon}{T}+\sqrt{4+(\frac{\varepsilon}{T})^2}}{2} < 1.
\end{equation}
Define $(\zetabd^*,\psibd^*)$ as the unique pair such that $(\zetabd^*,\psibd^*) \in \Cmc$ and
\begin{equation}\label{eq:515n}
\begin{aligned}
	\psi_1^*(t) & = 1 + a^* t - b^* \int_0^t (\zeta_1^*(s)-\zeta_2^*(s)) \,ds, \\
	\psi_i^*(t) & = x_i - b^* \int_0^t (\zeta_i^*(s)-\zeta_{i+1}^*(s)) \,ds \quad i \ge 2. 
\end{aligned}
\end{equation}
Intuitively, this means that the controlled arrival rate is $a^*$ and the controlled service rate at each server is $b^*$.
Existence and uniqueness of such a pair $(\zetabd^*,\psibd^*)$ follows from Lemma \ref{lem:special-well-posedness} with $\alpha \equiv a^*$ and $\theta \equiv b^*$.
Taking $\varphibd^*$ as
\begin{equation*}
	\varphi_0^*(s,y) := a^*, \quad \varphi_i^*(s,y) := b^* \one_{[0,\zeta_i^*(s)-\zeta_{i+1}^*(s))}(y) + \one_{[\zeta_i^*(s)-\zeta_{i+1}^*(s),1)}(y),
\end{equation*}
we see that \eqref{eq:psi1} and \eqref{eq:psii} hold, which means $\varphibd^* \in \Smc(\zetabd^*,\psibd^*)$.
Since $a^*>1>b^*$, we have $\zeta^*_1(t) \equiv 1$ and hence
\begin{equation*}
	\|\zetabd^*(T)\|_1 = \sum_{i=1}^\infty \psi^*_i(T) = \|\xbd\|_1 + a^*T - b^*\int_0^T\zeta_1^*(s)\,ds = \|\xbd\|_1 + \varepsilon.
\end{equation*}
This means $(\zetabd^*,\psibd^*) \in F_\varepsilon$ and
\begin{equation*}
	\Imc(F_\varepsilon) \le \Imc(\zetabd^*,\psibd^*) \le \sum_{i=0}^\infty \int_{\XT} \ell(\varphi_i^*(s,y))\,ds\,dy = T\ell(a^*) + T\ell(b^*).
\end{equation*}

Next we show that this is indeed optimal, namely $\Imc(F_\varepsilon) \ge T\ell(a^*) + T\ell(b^*)$.
For this, consider any $(\zetabd,\psibd) \in F_\varepsilon$ and $\varphibd \in \Smc(\zetabd,\psibd)$ with $\sum_{i=0}^\infty \int_{\XT} \ell(\varphi_i(s,y))\,ds\,dy < \infty$.
We claim that we can assume without loss of generality that $\|\zetabd(T)\|_1 \ge \|\xbd\|_1 + \varepsilon$ and $\zeta_1(t) = 1$ for all $t \in [0,T]$. 
To see this, let
\begin{equation*}
	\tau := \inf\{t \in [0,T] : \|\zetabd(t)\|_1 \ge \|\xbd\|_1 + \varepsilon\}
\end{equation*}
be the first time that $\zetabd$ meets the target level.
It suffices to show that there exist some $(\bar\zetabd,\bar\psibd) \in \Cmc$ and $\bar\varphibd \in \Smc(\bar\zetabd,\bar\psibd)$ such that $\zetabar_1(t) = 1$ for all $t \in [0,\tau]$, $\|\bar\zetabd(\tau)\|_1 \ge \|\zetabd(\tau)\|_1$ and
\begin{equation}
	\label{eq:special-cost}
	\sum_{i=0}^\infty \int_{[0,\tau]\times[0,1]} \ell(\varphibar_i(s,y))\,ds\,dy \le \sum_{i=0}^\infty \int_{[0,\tau]\times[0,1]} \ell(\varphi_i(s,y))\,ds\,dy,
\end{equation}
as one can simply follow $(\bar\zetabd,\bar\psibd)$ up to time $\tau$ and then switch to the law of large numbers limit trajectory afterwards.
Since $\ell(\cdot)$ is a convex function, by appealing to Jensen's inequality, we have
\begin{align*}
    \sum_{i=1}^\infty \int_{[0,\tau]\times[0,1]} \ell(\varphi_i(s,y))\,ds\,dy 
    & \ge \sum_{i=1}^\infty \int_{[0,\tau]\times[0,1]} \ell(\varphi_i(s,y)) \one_{[0,\zeta_i(s)-\zeta_{i+1}(s))}(y) \,ds\,dy \\
    & \ge \sum_{i=1}^\infty \int_{[0,\tau]\times[0,1]} \ell(\varphitil_i(s,y)) \one_{[0,\zeta_i(s)-\zeta_{i+1}(s))}(y) \,ds\,dy,
\end{align*}
where $\varphitil_i(s,y)$ is the average of $\varphi_i(s,y)$ over $y \in [0,\zeta_i(s)-\zeta_{i+1}(s))$ for each $i \in \Nmb$ and $s \in [0,T]$, namely
$$\varphitil_i(s,y) = \one_{[0,\zeta_i(s)-\zeta_{i+1}(s))}(y) \frac{\int_0^1\varphi_i(s,z)\one_{[0,\zeta_i(s)-\zeta_{i+1}(s))}(z)\,dz}{\zeta_i(s)-\zeta_{i+1}(s)} + \one_{[\zeta_i(s)-\zeta_{i+1}(s),1]}(y).$$
Also note that $(\varphi_0,\varphitil_1,\varphitil_2,\dotsc) \in \Smc(\zetabd,\psibd)$.
Therefore, without loss of generality, we can assume that $\varphi_i(s,y)$ is constant over $y \in [0,\zeta_i(s)-\zeta_{i+1}(s))$ and $1$ over
$[\zeta_i(s)-\zeta_{i+1}(s), 1]$, for each $i \in \Nmb$ and $s \in [0,T]$.
Let 
$$\theta(s,y) := \sum_{i=1}^\infty \varphi_i(s,y-(1-\zeta_i(s))) \one_{[1-\zeta_i(s),1-\zeta_{i+1}(s))}(y) + \one_{[0,1-\zeta_1(s))}(y)$$
be the ``master'' control of $\varphibd$.
Let $\thetabar(s,y) = \min\{1,\theta(s,y)\}$ and $\varphibar_0(s,y) = \max\{1,\varphi_0(s,y)\}$.
Then $\|\thetabar\|_\infty \le 1$ and
\begin{equation}
    \label{eq:special-bound}
    \int_{\Xt} \ell(\varphibar_0(s,y))\,ds\,dy \le \int_{\Xt} \ell(\varphi_0(s,y))\,ds\,dy < \infty, \quad \forall\,t\in[0,T],
\end{equation}
as $\ell(x)$ in decreasing in $0 \le x \le 1$.
Also note that $\thetabar(s,y)$ is continuous in a.e.\ $y$ for each $s \in [0,T]$.
It then follows from Lemma \ref{lem:special-well-posedness} (with $\alpha$ and $\theta$ there replaced by $\varphibar_0$ and $\thetabar$) that there exists a unique pair $(\bar\zetabd,\bar\psibd) \in \Cmc$ such that
\begin{align*}
	\psibar_1(t) & = x_1 + \int_{\Xt}\varphibar_0(s,y)\,ds\,dy  - \int_{\Xt}\one_{[1-\zetabar_1(s),1-\zetabar_2(s))}(y) \thetabar(s,y) \,ds\,dy, \\
	\psibar_i(t) & = x_i - \int_{\Xt}\one_{[1-\zetabar_i(s),1-\zetabar_{i+1}(s))}(y) \thetabar(s,y) \,ds\,dy, \quad i \ge 2.
\end{align*}		
In particular, $\bar\varphibd \in \Smc(\bar\zetabd,\bar\psibd)$ where
\begin{equation*}
	\varphibar_i(s,y) = \thetabar(s,y+1-\zetabar_i(s)) \one_{[0,\zetabar_i(s)-\zetabar_{i+1}(s))}(y) + \one_{[\zetabar_i(s)-\zetabar_{i+1}(s),1)}(y), \quad i \ge 1.
\end{equation*}
Since $\varphibar_0 \ge 1$ and $\thetabar \le 1$, we see that $\psibar_1$ is non-decreasing and hence $\zetabar_1(t) = 1$ for all $t \in [0,\tau]$.
Also, since $\ell(x)$ is increasing in $x \ge 1$, the construction of $\thetabar$ and $\bar\varphibd$ guarantees
\begin{align*}
    & \sum_{i=1}^\infty \int_{[0,\tau]\times[0,1]} \ell(\varphi_i(s,y))\,ds\,dy 
    = \sum_{i=1}^\infty \int_{[0,\tau]\times[0,1]} \ell(\varphi_i(s,y)) \one_{[0,\zeta_i(s)-\zeta_{i+1}(s))}(y) \,ds\,dy \\
    & = \int_{[0,\tau]\times[0,1]} \ell(\theta(s,y))\,ds\,dy 
    \ge \int_{[0,\tau]\times[0,1]} \ell(\thetabar(s,y))\,ds\,dy \\
    & = \sum_{i=1}^\infty \int_{[0,\tau]\times[0,1]} \ell(\varphibar_i(s,y)) \one_{[0,\zetabar_i(s)-\zetabar_{i+1}(s))}(y) \,ds\,dy 
    = \sum_{i=1}^\infty \int_{[0,\tau]\times[0,1]} \ell(\varphibar_i(s,y))\,ds\,dy. 
\end{align*}
This and \eqref{eq:special-bound} give \eqref{eq:special-cost}.
It now remains to show $\|\bar\zetabd(\tau)\|_1 \ge \|\zetabd(\tau)\|_1$.
Note that
\begin{align*}
	& \|\bar\zetabd(\tau)\|_1 - \|\zetabd(\tau)\|_1 = \sum_{i=1}^\infty \psibar_i(\tau) - \sum_{i=1}^\infty \psi_i(\tau) \\
	& = \int_0^\tau \left(\int_0^1 [\varphibar_0(s,y) - \thetabar(s,y)\one_{[1-\zetabar_1(s),1)}(y) - \varphi_0(s,y) + \theta(s,y)\one_{[1-\zeta_1(s),1)}(y)]\,dy \right)ds \\
	& \ge \int_{s \in (0,\tau) : \zeta_1(s) < 1} \left(\int_0^1 [\varphibar_0(s,y) - \thetabar(s,y) - \varphi_0(s,y) + \theta(s,y)\one_{[1-\zeta_1(s),1)}(y)]\,dy \right)ds,
\end{align*}
where the last line follows on noting that $\varphibar_0 \ge \varphi_0$, $\thetabar \le \theta$, and $\zetabar_1(s) \equiv 1$ and hence the inside integral is non-negative whenever $\zeta_1(s)=1$.
Since $\{s \in (0,\tau) : \zeta_1(s) < 1\}$ is an open set, we can write
$$\{s \in (0,\tau) : \zeta_1(s) < 1\} = \bigcup_{k=1}^\infty E_k$$
for disjoint intervals $E_k=(a_k,b_k)$ with $\zeta_1(a_k)=\zeta_1(b_k)=1$.
It then suffices to show
$$\int_{a_k}^{b_k} \left(\int_0^1 [\varphibar_0(s,y) - \thetabar(s,y) - \varphi_0(s,y) + \theta(s,y)\one_{[1-\zeta_1(s),1)}(y)]\,dy \right)ds \ge 0$$
for each $k$.
Since $\varphibar_0 \ge 1 \ge \thetabar$, it suffices to show
$$\int_{a_k}^{b_k} \left(\int_0^1 [\varphi_0(s,y) - \theta(s,y)\one_{[1-\zeta_1(s),1)}(y)]\,dy \right)ds \le 0.$$
But the left hand side is simply
$$\sum_{i=1}^\infty [\psi_i(b_k)-\psi_i(a_k)] \le \psi_1(b_k)-\psi_1(a_k) = \zeta_1(b_k) - \zeta_1(a_k) = 0,$$
where the first equality follows as $\zeta_1(s)<1$ for $s \in E_k=(a_k,b_k)$ and hence there is no contribution from the reflection terms over this interval.
Therefore we have verified $\|\bar\zetabd(\tau)\|_1 \ge \|\zetabd(\tau)\|_1$ and hence the claim holds. 

Now fix $\sigma \in (0,1)$.
Consider any $(\zetabd,\psibd) \in F_\varepsilon$ and $\varphibd \in \Smc(\zetabd,\psibd)$ with 
$$\sum_{i=0}^\infty \int_{\XT} \ell(\varphi_i(s,y))\,ds\,dy \le \Imc(\zetabd,\psibd) + \sigma < \infty,$$ 
$\|\zetabd(T)\|_1 \ge \|\xbd\|_1 + \varepsilon$ and $\zeta_1(t) = 1$ for all $t \in [0,T]$.
Then
\begin{align*}
	& \sum_{i=0}^\infty \int_{\XT} \ell(\varphi_i(s,y))\,ds\,dy \\
	& \ge \int_{\XT} \left[\ell(\varphi_0(s,y)) + \sum_{i=1}^\infty \ell(\varphi_i(s,y)) \one_{[0,\zeta_i(s)-\zeta_{i+1}(s))}(y) \right]ds\,dy \\
	& \ge T \ell\left(\frac{1}{T} \int_{\XT} \varphi_0(s,y)\,ds\,dy\right) + T \ell\left(\frac{1}{T} \int_{\XT} \sum_{i=1}^\infty \varphi_i(s,y) \one_{[0,\zeta_i(s)-\zeta_{i+1}(s))}(y) \,ds\,dy\right),
\end{align*}
where the third line uses Jensen's inequality and the fact that $\zeta_1(s) \equiv 1$.
This quantity can be further bounded from below by
\begin{equation*}
	T \inf\{\ell(a) + \ell(b) : a,b \ge 0, a-b=c, c \ge \frac{\varepsilon}{T} \},
\end{equation*}
where the constraint $c \ge \varepsilon/T$ follows on observing that
$\varepsilon \le \|\zetabd(T)\|_1-\|\xbd\|_1 = \sum_{i=1}^\infty [\psi_i(T)-\psi_i(0)]$.
Using Lagrange multiplies one finds that given $c \ge \frac{\varepsilon}{T}$, the above infimum is achieved at
$$\ahat=\frac{c+\sqrt{c^2+4}}{2}, \quad \bhat=\frac{1}{\ahat} = \frac{-c+\sqrt{c^2+4}}{2}$$
with value
$f(c) := \ell(\ahat)+\ell(\bhat) = \ell(\ahat)+\ell(1/\ahat).$
Note that
$$\frac{df}{dc} = \log(\ahat) \frac{d\ahat}{dc} + \frac{\log(\ahat)}{\ahat^2} \frac{d\ahat}{dc}.$$
Since $\log(\ahat) \ge 0$ and $\frac{d\ahat}{dc} \ge 0$, we see that the infimum over $c \ge \varepsilon/T$ is finally achieved at $\chat=\varepsilon/T$.
This is exactly the choice in \eqref{eq:optimizer} for the candidate optimizer.
Since $\sigma \in (0,1)$ is arbitrary, we have that
$$\Imc(F_\varepsilon) = \Imc(\zetabd^*,\psibd^*) = \sum_{i=0}^\infty \int_{\XT} \ell(\varphi_i^*(s,y))\,ds\,dy = T\ell(a^*) + T\ell(b^*).$$

Note that from the form of $a^*$ and $b^*$ in \eqref{eq:optimizer}, $\varepsilon \mapsto \Imc(F_\varepsilon)$ is continuous in $(0,1)$.
Next, note that for $\delta \in (0, \varepsilon)$,
$F_{\varepsilon+\delta} \subset G_\varepsilon \subset F_\varepsilon \subset F_{\varepsilon-\delta}$.
Since $\xbd^n \to \xbd$ as $n \to \infty$ we have from the LDP in Theorem \ref{thm:mainResult}, that 
\begin{equation*}
	 - \Imc(F_{\varepsilon+\delta}) \le \liminf_{n\to\infty}\frac{1}{n}\log(\Pmb(G_\varepsilon^n)) \le
	\limsup_{n\to\infty}\frac{1}{n}\log(\Pmb(F_\varepsilon^n)) \le   - \Imc(F_{\varepsilon-\delta}).
\end{equation*}
The first statement in Theorem \ref{thm:large-customers} now follows on sending $\delta \to 0$ in the above display.
The second statement follows from the observations that $\sqrt{4+x^2}=2+o(x)$ and $\ell(1+x)=\frac{x^2}{2} + o(x^2)$ as $x\to0$. \qed

\appendix

\section{Comments on the proof of Lemma \ref{lem:uniqueness}(a)}
\label{sec:appendix}

In this appendix we briefly explain how the product topology in \cite[Lemma 5.1(a)]{BudhirajaFriedlanderWu2021many} can be improved to the $\|\cdot\|_{1,\infty}$ norm in Lemma \ref{lem:uniqueness}(a) of the current work. 
The key observation is that the $\|\cdot\|_\infty$ norm in \cite[Lemma 5.2(iii)]{BudhirajaFriedlanderWu2021many} can be replaced with the $\|\cdot\|_{1,\infty}$ norm which then immediately yields the strengthened form of Lemma 5.1(a). 
Here we note that the statement of \cite[Lemma 5.2]{BudhirajaFriedlanderWu2021many} contains an error which has been corrected in \cite[Lemma 5.2*]{BudhirajaFriedlanderWu2025errata}, however that does not significantly affect the discussion below. 

The overall idea is that the proof of \cite[Lemma 5.2]{BudhirajaFriedlanderWu2021many} proceeds by a series of approximations of a given given $(\zetabd^*,\psibd^*)\in\Cmc$ that only affect finitely many coordinates at each stage and thus controlling the $\|\cdot\|_{1,\infty}$ norm is as easy as controlling the $\|\cdot\|_{\infty}$ norm, in fact many of the estimates used in the proof of \cite[Lemma 5.2]{BudhirajaFriedlanderWu2021many} are obtained by bounding the latter by the former.
Specifically, the changes needed are as follows.

\begin{itemize}
\item 
    In the statement of \cite[Lemma 5.3(b)]{BudhirajaFriedlanderWu2021many}, $\|\cdot\|_{\infty}$ can be replaced by $\|\cdot\|_{1,\infty}$. This is done by observing that the second displayed equation from bottom on page 2404 can be changed as
	$$\|(\zetabd,\psibd) - (\tilde\zetabd,\tilde\psibd)\|_{1,\infty}
    \le \sum_{k=1}^{K-1} \|\zeta_k -\tilde \zeta_k\|_{\infty} +  \sum_{k=1}^{K-1} \|\psi_k -\tilde \psi_k\|_{\infty} \le \frac{\sigma}{4} + \frac{3\sigma}{4} = \sigma.$$
\item 
    \cite[Lemma 5.4]{BudhirajaFriedlanderWu2021many} is a direct consequence of \cite[Lemma 5.3]{BudhirajaFriedlanderWu2021many} and so the $\|\cdot\|_{\infty}$ in the statement of Lemma 5.4 can be replaced by the $\|\cdot\|_{1,\infty}$ norm by using the above strengthened form of Lemma 5.3.
\item 
    Finally, in the statement of \cite[Lemma 5.2]{BudhirajaFriedlanderWu2021many}, $\|\cdot\|_{\infty}$ can be replaced by the $\|\cdot\|_{1,\infty}$ norm as follows.
    \begin{itemize}
    \item 
        In the second line from bottom on page 2405, $\|\cdot\|_{\infty}$ can be replaced by the $\|\cdot\|_{1,\infty}$ norm as this is just re-stating (the above strengthened form of) Lemma 5.4. 
        Using this and (5.30) in \cite{BudhirajaFriedlanderWu2021many}, the first displayed equation on page 2408 can be replaced as
    	\begin{align*}
    		\|(\bar\zetabd^{new},\bar\psibd^{new}) - (\tilde\zetabd,\tilde\psibd)\|_{1,\infty} & \le \|(\bar\zetabd^{new},\bar\psibd^{new}) - (\bar\zetabd,\bar\psibd)\|_{1,\infty} + \|(\bar\zetabd,\bar\psibd) - (\tilde\zetabd,\tilde\psibd)\|_{1,\infty} \\
    		& \le \sum_{i=1}^{\bar N} \frac{\sigma}{8 \bar N} + \frac{\sigma}{16} \le \frac{3\sigma}{16}.
    	\end{align*}
    \item 
        Using the estimate above 
    	(5.37) in \cite{BudhirajaFriedlanderWu2021many}   and observing that $\bar \zetabd$ and $\zetabd$ differ only in the $(K+1)$-th coordinate, the estimate in 
        (5.37) of \cite{BudhirajaFriedlanderWu2021many} can be written as
    	$$\|\zetabd-\bar\zetabd\|_{1,\infty} \le \frac{\sigma}{4\Nbar}.$$
    \item 
    	The first display on page 2411, in fact gives a bound on the $\|\cdot \|_{1, \infty}$ norm and says that (cf. \cite{BudhirajaFriedlanderWu2025errata})
    	\begin{align*}
    		\|\psibd-\bar\psibd\|_{1,\infty} & \le \frac{\sigma}{16}.
        \end{align*}
    \item 
        Combining the last two estimates, the second display on page 2411 can be replaced as
        \begin{align*}
    		\|(\zetabd,\psibd)-(\bar\zetabd,\bar\psibd)\|_{1,\infty} & \le \frac{\sigma}{16} + \frac{\sigma}{4\Nbar} \le \frac{5\sigma}{16}.
    	\end{align*}
    \item 
        The argument below the above estimate in \cite{BudhirajaFriedlanderWu2021many} is not needed any more, as explained in \cite{BudhirajaFriedlanderWu2025errata}, and hence no further changes of norms are needed. 
    \end{itemize}
    These changes complete the proof of \cite[Lemma 5.2]{BudhirajaFriedlanderWu2021many} with $\|\cdot\|_{\infty}$ replaced by the $\|\cdot\|_{1,\infty}$ norm.
\item 
    The strengthened form of Lemma 5.2 immediately yields the strengthened form of Lemma 5.1 as stated in the current work.
\end{itemize}

\subsection*{Acknowledgments}
AB was partially supported by NSF DMS-2152577, NSF  DMS-2134107, NSF DMS-2506010.\\
RW was partially supported by NSF DMS-2308120 and Simons Foundation travel grant MP-TSM-00002346. 



\bigskip
\noindent \scriptsize{\textsc{ 
\begin{minipage}{0.5\linewidth}
A. Budhiraja \newline
Department of Statistics and Operations Research
\newline
University of North Carolina\newline
Chapel Hill, NC 27599, USA\newline
email:  budhiraj@email.unc.edu     
\end{minipage}
\hfill
\begin{minipage}{0.4\linewidth}
R. Wu\newline
Department of Mathematics,\newline
 Iowa State University\newline
 Ames,  IA 50011, USA\newline
email: ruoyu@iastate.edu 
\end{minipage}}
}
\end{document}